\documentclass[10pt]{article}
\usepackage[latin1]{inputenc}

\usepackage{enumitem}
\usepackage{amsthm}
\usepackage{amsmath,amssymb,wasysym,amsbsy}
\usepackage{bm}
\usepackage{graphicx}
\usepackage{subfigure}
\usepackage{multicol,array,lscape}
\usepackage{nonfloat}
\usepackage[small,bf]{caption2} 
\usepackage{breqn}
\usepackage[dvipsnames,svgnames,x11names]{xcolor}
\usepackage{pst-all}
\usepackage{pst-3dplot}
\usepackage{authblk}
\usepackage{cancel} 
\usepackage{soul} 
\usepackage[%
  colorlinks=true,%
	linkcolor	=blue,%
	citecolor =blue,%
  urlcolor=blue%
]{hyperref}
\usepackage[left=3cm,right=3cm, top=3.3cm, bottom=3cm]{geometry}
\newtheorem{theo}{Theorem}[section]
\newtheorem{lemm}[theo]{Lemma}

\newtheorem{defi}[theo]{Definition}
\newtheorem{remark}[theo]{Remark}

\def\x        {\textit {\textbf{x}}}

\numberwithin{equation}{section}

\title{Numerical analysis for a chemotaxis-Navier-Stokes system}
\author[a]{Abelardo Duarte-Rodr\'{\i}guez}
\author[b]{Mar\'{\i}a  A. Rodr\'{\i}guez-Bellido}
\author[a]{Diego A. Rueda-G\'omez}
\author[a]{\'Elder J. Villamizar-Roa\thanks{Corresponding author.
\href{mailto:jvillami@uis.edu.co}{jvillami@uis.edu.co}  (E. J. Villamizar-Roa).}}
\affil[a]{Universidad Industrial de Santander, Escuela de Matem\'{a}ticas, A.A. 678, Bucaramanga, Colombia.}
\affil[b]{Departamento de Ecuaciones Diferenciales y An\'alisis Num\'erico and IMUS, Universidad
de Sevilla, Facultad de Matem\'aticas, C/ Tarfia, S/N, 41012 Sevilla, Spain.}

\DeclareRobustCommand{\uvec}[1]{{%
  \ifcsname uvec#1\endcsname
     \csname uvec#1\endcsname
   \else
    \bm{\hat{\mathbf{#1}}}%
   \fi
}}
\date{}
\begin{document}
\maketitle
\begin{abstract}
In this paper we develop a numerical scheme for approximating a $d$-dimensional chemotaxis-Navier-Stokes system, $d=2,3$, modeling cellular swimming in incompressible fluids. This model describes the chemotaxis-fluid interaction in cases where the chemical signal is consumed with a rate proportional to the amount of organisms. We construct numerical approximations based on the Finite Element method and analyze some error estimates and convergence towards weak solutions. In order to construct the numerical scheme, we use a splitting technique to deal with the 
chemo-attraction term in the cell-density equation, leading to introduce a new variable given by the gradient of the chemical concentration. Having the equivalent model, we consider a fully discrete Finite Element approximation which is 
well-posed and it is mass-conservative.  We obtain uniform estimates and analyze the convergence of the scheme. Finally, we present some numerical simulations to verify the 
good behavior of our scheme, as well as to check numerically the error estimates proved in our theoretical analysis.
\vspace{0.3cm}

\noindent{\bf Keywords.} Chemotaxis-Navier-Stokes system, finite elements, convergence rates, error estimates.\vspace{0.3cm}

\noindent{\bf AMS subject classifications.}  {35Q35; 35Q92; 92C17; 65M12; 65M15; 65M60}
\end{abstract}

\section{Introduction}
 Chemotaxis is the oriented movement of cells towards the concentration gradient of certain chemicals in their environment. In particular, when the movement of cells is directed towards the increasing concentration of a signal, the phenomenon is known as chemotaxis by atracction. This kind of phenomena, which play an outstanding role in a large range of biological applications, are modeled, in their simplest form, by the Keller-Segel system. However, some experimental studies, as reported in \cite{Hill,tuval2005bacterial}, reveal that the chemotactic motion in liquid environments affects substantially the migration of cells. Some examples of these facts are the phenomenon of broadcast spawning, the pattern generation and spontaneous emergence of turbulence in populations of aerobic bacteria suspended in sessile drops of water \cite{Tao_2015_boun}. This kind of interaction can be modeled through a chemotaxis-Navier-Stokes system. This model is given by the following system of PDEs:
\begin{equation}\label{KNS}
\left\{
\begin{array}{lc}
\eta_{t}+\mathbf{u} \cdot\nabla \eta=D_n\Delta \eta-\chi \nabla\cdot( \eta\nabla c), &\\
c_{t}+\mathbf{u} \cdot\nabla c=D_c\Delta c- \gamma \eta c,  &\\
\rho \left(\mathbf{u}_{t}+(\mathbf{u}\cdot\nabla)\mathbf{u} \right)= D_{\mathbf{u}} \Delta \mathbf{u}-\nabla\pi+ \eta \nabla \phi, &\\
\nabla\cdot \mathbf{u}=0, &
\end{array}
\right.
\end{equation}
where $\eta = \eta(\x, t),$ $c =c(\x, t),$ $\pi(\x,t)$ and $\mathbf{u}(\x,t)$ denote respectively the cell density, the concentration of an attractive chemical signal, the hydrostatic pressure, and the velocity field of the fluid at position $\x\in \Omega\subset\mathbb{R}^d,$ $d=2,3,$ and time $t\in(0,T]$. This model describes the interaction between a type of cells (e.g., bacteria), and a chemical signal 
which is consumed with a rate proportional to the amount of organisms. The cells and chemical substance are transported by a viscous incompressible fluid under the influence of a force due to the aggregation of cells. The equation for the velocity field $\mathbf{u}$ is described by the incompressible Navier-Stokes system with forcing term given by $-\eta \nabla \phi,$ which represents the effects due to density variations caused by cell aggregation. The parameters $\chi,D_n,D_c,\rho$ and $D_{\mathbf{u}}$ are positive constants that represent the chemotactic coefficient, the cell diffusion constant, the chemical diffusion constant, the fluid density and the viscosity of fluid, respectively.\\ 

System (\ref{KNS}) is completed with the following initial and boundary data:
\begin{equation}\label{initialdata}
\left\{
\begin{array}{lc}
\left[\eta(\x,0),c(\x,0),\mathbf{u}(\x,0)\right]=\left[\eta_0(\x),c_0(\x),\mathbf{u}_0(\x)\right],\ \x\in\Omega,\\[.3cm]
\frac{\partial \eta(\x,t)}{\partial \boldsymbol{\nu}}=\frac{\partial c(\x,t)}{\partial \boldsymbol{\nu}}=0, \quad \mathbf{u}(\x,t)=0,\quad \x\in\partial\Omega,\quad t\in(0,T).
\end{array}
\right.
\end{equation}
The mathematical understanding of the existence and uniqueness of a solution for (\ref{KNS})-(\ref{initialdata}) is quite challenging, due the coupling between the Navier-Stokes equations and the chemotaxis system. However, there are several results  of existence, uniqueness, regularity and qualitative properties of the solutions for system (\ref{KNS})-(\ref{initialdata}) and related models (see for instance \cite{Duarte, Jiang,lankeit2016long,winkler2012global, winkler2016global,Winkler3, Zhang} and references therein). In \cite{winkler2012global}, it was proved the existence of global classical solutions in  two-dimensional bounded convex domains. The results of \cite{winkler2012global} were extended to nonconvex domains in \cite{Jiang}. Results of convergence of classical solutions to the corresponding stationary model were analyzed in \cite{Winkler3, Zhang}. In \cite{winkler2016global}, it was proved the existence of global weak solutions in bounded three-dimensional convex domains, considering the Stokes system instead of the Navier-Stokes equations, that is, neglecting the nonlinear term $(\mathbf{u}\cdot \nabla)\mathbf{u}$ in the fluid equation. The existence of global weak solutions for the full three-dimentional  Chemotaxis Navier-Stokes model in bounded convex domains was obtained in \cite{winkler2016global}. In \cite{winkler2016global} the author also proved that any eventual energy solution becomes smooth after some waiting time. In \cite{lankeit2016long} the author proved the existence of weak solutions for general bounded domains of $\mathbb{R}^3$ for the chemotaxis-Navier-Stokes system with logistic source. In the same paper, the author also proved that after some waiting time, the weak solutions become smooth and converge to some steady state. Results of global existence of mild solutions in bounded domains of $\mathbb{R}^d$ , $d=2,3,$ for small initial data in $L^p$-spaces were obtained in \cite{Duarte}. However, to the best of our knowledge, there is no convergence numerical analysis for the solutions of (\ref{KNS})-(\ref{initialdata}). We only know some numerical simulations to investigate the patterns formation and predict numerically the nonlinear dynamic of the chemotaxis-fluid system (see for instance \cite{Chertock,Deleuze, Karimi, Lee, tuval2005bacterial}).\\ 

This paper is focused on the analysis of a numerical method to approximate the solutions of (\ref{KNS})-(\ref{initialdata}). As far as we know, there is no numerical analysis related to the convergence and error estimates of approximations for the weak solutions of (\ref{KNS})-(\ref{initialdata}). Even more, there are only a few works about numerical analysis for Keller-Segel system (i.e., with chemo-attraction and linear production)  and assuming that there is no interaction with the fluid \cite{Epshteyn, Filbet, Marrocco, Saito1, Saito2}. In \cite{Filbet} the author analyzed the existence of discrete solutions and the convergence of a finite volume scheme. In \cite{Saito1, Saito2}, was obtained some error estimates for a conservative Finite Element (FE) approximation. In \cite{Epshteyn} were proved some error estimates for a fully discrete discontinuous FE method. A mixed FE approximation was analyzed in \cite{Marrocco}. On the other hand, in the same framework of Keller-Segel system some previous energy stable numerical schemes have also been analyzed (cf. \cite{Bessemoulin-Chatard}). In addition, unconditionally energy stable time-discrete numerical schemes and fully discrete FE schemes for a chemo-repulsion model with quadratic production has been analyzed in \cite{GMD2,GMD3}. Some unconditionally energy stable fully discrete schemes for a parabolic repulsive-productive chemotaxis model (with linear production term) was recently analyzed in  \cite{GMD1}. In our case, the difficulties to deal with the numerical analysis of (\ref{KNS})-(\ref{initialdata}) come from the strong coupling nonlinear  term $\nabla\cdot(\eta \nabla c)$ and the transport terms $\mathbf{u}\cdot\nabla \eta,$ $\mathbf{u}\cdot\nabla c$ and $\mathbf{u}\cdot\nabla \mathbf{u}.$ Therefore, in order to deal the chemotaxis term $\nabla\cdot(\eta \nabla c)$ we introduce a new variable given by the gradient of the chemical concentration, that is,  $\boldsymbol{\sigma} = \nabla c,$ which allow us to control the strong regularity request by the system; on the other hand, the transport terms have the difficulty that  the corresponding discrete forms do not preserve the same properties of the continuous case. We concentrate our analysis by introducing a splitting mixed finite element method for approximating the weak solutions of  (\ref{KNS})-(\ref{initialdata}) and analyze some error estimates.  This scheme is 
well-posed and it is mass-conservative.  We obtain uniform estimates and analyze the convergence towards weak solutions. We also present some numerical simulations in order to validity the numerical results.\\ 

The plan of this paper is as follows. In Section 2, we first establish some basic notations and recall some existence and uniqueness results of (\ref{KNS})-(\ref{initialdata}) in the continuous case. We also define the weak formulation of (\ref{KNS})-(\ref{initialdata}), which will be used to construct the numerical approximation. In Section 3, we construct  a mass-conservative numerical approximation for the weak solutions of the chemotaxis-fluid system (\ref{KNS}) with initial and boundary data (\ref{initialdata}), by using FE approximations in space and finite differences in time, and give some preliminary results concerning to the FE spaces. In Section 4, we analyze the well-posedness and the mass-conservation property (see (\ref{mass01b}) below) for the numerical scheme; and following an inductive procedure, in Section 5, we obtain some uniform estimates for any solution of scheme (\ref{scheme1}), subsequently required in the convergence analysis. The analysis of Section 5 is focused on the two-dimensional case. In Section 6, we make some comments about the extensions of these results to the three-dimensional case. Finally, in Section 7, we provide some numerical simulations in agreement with the numerical results.
\section{The continuous problem}
In this section, we establish a variational formulation  of (\ref{KNS})-(\ref{initialdata}) which will be used to construct the numerical scheme. In order to deal with the chemotaxis term in the cell density equation, we introduce a splitting variational formulation for the chemoatractant concentration, which yields to consider an auxiliar variable representing the gradient of the chemical signal. In this section, after to establish the basic notation which will be used, we recall  some existence results for two and three dimensional bounded domains, and define a variational formulation of the continuous problem.

 \subsection{Notations}
We consider the standard Sobolev and Lebesgue spaces $W^{k,p}(\Omega)$ and $L^p(\Omega),$ with respective norms $\Vert \cdot\Vert_{W^{k,p}}$ and $\Vert \cdot\Vert_{L^p}.$ In particular, we denote $W^{k,2}(\Omega)=H^k(\Omega).$ Also, $W_0^{1,p}(\Omega)$ denotes the elements of $W^{1,p}(\Omega)$ with trace zero on the boundary of $\Omega.$  The $L^2(\Omega)$-inner product will be represented by $(\cdot,\cdot).$ Corresponding Sobolev spaces of vector valued functions will be denoted by ${\bf W}^{k,p} (\Omega),$ ${\bf L}^{p} (\Omega),$ and so on.  Also, we will use the following function spaces
\begin{eqnarray*}
\mathbf{V}&:=&\{\mathbf{v} \in \mathbf{H}^1_0(\Omega):\nabla\cdot \mathbf{v}=0 \: \text{ in} \: \Omega\},\\
\mathbf{H}^{1}_{\sigma}(\Omega)&:=&\{\mathbf{v}\in \mathbf{H}^{1}(\Omega): \mathbf{v}\cdot \boldsymbol{\nu}=0 \mbox{ on } \partial\Omega\},\\
{L}_0^2(\Omega)&:=& \left\{ p \in L^2(\Omega): \int_{\Omega}p =0 \right\},\\
\widehat{H}^1(\Omega)&:=&H^1(\Omega) \cap L^2_0(\Omega)=\left\{ w \in H^1(\Omega): \int_{\Omega}w =0 \right\},
\end{eqnarray*}
where $\boldsymbol{\nu}$ denotes the unit outward normal vector to the boundary. We will use the following equivalent norms in $H^1(\Omega)$  and ${\bf H}_{\sigma}^1(\Omega)$, respectively (see \cite{necas} and \cite[Corollary 3.5]{Nour}):
\begin{equation}\label{EQu}
\Vert u \Vert_{H^1}^2=\Vert \nabla u\Vert_{0}^2 + \left( \int_\Omega u\right)^2, \ \ \forall u\in H^1(\Omega),
\end{equation}
\begin{equation}\label{EQs}
\Vert {\boldsymbol\sigma} \Vert_{H^1}^2=\Vert {\boldsymbol\sigma}\Vert_{0}^2 + \Vert \mbox{rot }{\boldsymbol\sigma}\Vert_0^2 + \Vert \nabla \cdot {\boldsymbol\sigma}\Vert_0^2, \ \ \forall {\boldsymbol\sigma}\in \mathbf{H}^{1}_{\sigma}(\Omega).
\end{equation}
We recall the Poincar\'e inequality 
\begin{equation}\label{PIa}
\|\mathbf{v}\|_{{H}^1} \leq C_{P}\|\nabla \mathbf{v}\|_{{L}^2}, \ \ \forall \mathbf{v} \in \mathbf{H}^1_0(\Omega),
\end{equation}
for some constant $C_{P}>0$ independent of $\mathbf{v}$. Also, we will use the classical interpolations inequalities
\begin{equation}\label{in2D}
\Vert u\Vert_{L^4}\leq C\Vert u\Vert_{L^2}^{1/2}\Vert u\Vert_{H^1}^{1/2} \  \ \forall u\in H^1(\Omega) \ \ \mbox{ (in 2D domains)},
\end{equation}
\begin{equation}\label{in3D}
\Vert u\Vert_{L^3}\leq C\Vert u\Vert_{L^2}^{1/2}\Vert u\Vert_{L^6}^{1/2} \  \ \forall u\in H^1(\Omega) \ \ \mbox{ (in 3D domains)}.
\end{equation}
Finally, the letter $C$ denotes different positive constants (independent of discrete parameters) which may change from line to line (or even within the same line).

\subsection{Variational formulation}
In this subsection we give a variational formulation for (\ref{KNS})-(\ref{initialdata}). We start by recalling a result in \cite{winkler2012global} which provides the existence and uniqueness of global classical solution for (\ref{KNS})-(\ref{initialdata}) in bounded and convex domains of $\mathbb{R}^2$ with boundary $\partial\Omega$ smooth enough.  Let us assume that the initial data satisfy the following regularity conditions:
\begin{equation}\label{data02}
\left\{
\begin{array}{lc}
\eta_0 \in C^0(\overline{\Omega}), \ \ n_0>0 \ \ \text{in } \: \overline{\Omega}, &\\
c_0 \in W^{1,q}(\Omega) \: \text{ for some }  q>2, \ \ c_0>0 \ \ \text{in } \: \overline{\Omega}, &\\
\mathbf{u}_0 \in D(A^{\alpha}) \ \ \text{ for some } \alpha \in (\frac{1}{2},1), \nabla\phi\in C^1(\overline{\Omega}),&
\end{array}
\right.
\end{equation}
where $A^\alpha$ denotes the fractional Stokes operator with domain $D(A^{\alpha}).$ Then, the following theorem holds.
\begin{theo}
Let $\Omega\subset \mathbb{R}^2$ be a bounded domain with smooth boundary, and suppose that $\eta_0,c_0, \mathbf{u}_0,\phi$ satisfy (\ref{data02}). Then (\ref{KNS})-(\ref{initialdata}) possesses a classical solution which is global in time. This solution is unique, up to addition of constants to the pressure $\pi$, and for all $T\in(0,\infty)$ the solution has the following regularity properties 
\begin{equation}\label{regul}
\left\{
\begin{array}{lc}
\eta\in C([0,T);L^2(\Omega))\cap L^\infty((0,T);C(\bar{\Omega}))\cap C^{2,1}(\bar{\Omega}\times(0,T)),\\
c\in C([0,T);L^2(\Omega))\cap L^\infty((0,T);W^{1,q}({\Omega}))\cap C^{2,1}(\bar{\Omega}\times(0,T)),\\
\mathbf{u}\in C([0,T);L^2(\Omega))\cap L^\infty((0,T);D(A^\alpha))\cap C^{2,1}(\bar{\Omega}\times(0,T)),\\
\pi\in L^1((0,T);W^{1,2}).
\end{array}
\right.
\end{equation}
\end{theo}
\begin{proof}
The proof is given in \cite[Theorem 1.1]{winkler2012global}. Indeed, the proof of \cite{winkler2012global} requires that $\Omega$ be a bounded convex domain. This condition comes from an integration term on the boundary which appears dealing with the differentiation of $\int_\Omega\vert \nabla \sqrt c\vert^2.$ That boundary term takes the form $\int_{\partial \Omega}\frac{\partial}{\partial \nu}\vert \nabla c\vert^2dS.$ Thus, if $\Omega$ is convex, the last boundary integral turns out to be non-positive and thus, this term can be neglected. However, the convexity condition can be removed by using Lemma 4.2 in \cite{Mizoguchi}, as used in \cite{Jiang,lankeit2016long}.
\end{proof}

For the three-dimensional setting of (\ref{KNS})-(\ref{initialdata}) and large initial data, Winkler in \cite{winkler2016global} proved the existence of weak solutions (for bounded convex domains). The notion of weak solution considered in \cite{winkler2016global} is the following one:
\begin{defi}\label{weak}(\textit{Weak solution})
A weak solution of (\ref{KNS})-(\ref{initialdata}) is a triple $[n,c,{u}]$ of functions such that
$\eta\in {L^{1}([0,T],W^{1,1}}(\Omega)),$  $c\in {L^{1}([0,T],W^{1,1}}(\Omega)),$  $\mathbf{u}\in L^1([0,T],W^{1,1}),$ verifying $\eta\geq 0$ and $c\geq 0$ a.e. in $\Omega\times (0,T),$ $nc\in L^1(0,T; L^1(\Omega)),$ $\mathbf{u} \otimes \mathbf{u}\in L^1(0,T; L^1(\Omega)),$ $\eta\nabla c, \eta c$ and $\eta \mathbf{u}\in L^1(0,T; L^1(\Omega)),$ $\nabla\cdot \mathbf{u}=0$ a.e in $\Omega\times (0,T),$ and such that
\begin{eqnarray*}
-\int_0^T\int_\Omega \eta\varphi_t-\int_\Omega \eta_0\varphi(\cdot,0)-\int_0^T\int_\Omega \eta{\mathbf{u}}\cdot\nabla\varphi=-\int_0^T\int_\Omega \nabla \eta\cdot \nabla  \varphi
+\chi\int_0^T\int_\Omega \eta\nabla c\cdot \varphi,\\
-\int_0^T\int_\Omega c\varphi_t-\int_\Omega c_0\varphi(\cdot,0)-\int_0^T\int_\Omega c\mathbf{u}\cdot\nabla\varphi+\int_0^T\int_\Omega \nabla c\cdot\nabla  \varphi=-\gamma\int_0^T\int_\Omega \eta c\varphi,\\
-\int_0^T\int_\Omega {\mathbf{u}}{\psi}_t-\int_\Omega {\mathbf{u}}_0{\psi}(\cdot,0)-\int_0^T\int_\Omega {\mathbf{u}}\otimes{\mathbf{u}}\cdot\nabla{\psi}+\int_0^T\int_\Omega \nabla {\mathbf{u}}\cdot \nabla  {\psi}
=-\int_0^T\int_\Omega \eta\nabla\phi\cdot\nabla{\psi},
\end{eqnarray*}
for all $\varphi\in C_0^\infty([0,T]\times \bar{\Omega})$ with $\varphi(T)=0,$ and ${\psi}$ in $C_0^\infty([0,T]\times \bar{\Omega})^3$ with ${\psi}(T)=0.$
\end{defi}
By assuming $\phi\in W^{2,\infty}({\Omega})$ and the initial data $[n_0,c_0,\mathbf{u}_0]$ satisfying 
\begin{equation}\label{initial1b}
\left\{
\begin{array}{lc}
\eta_0\in L \log L(\Omega)\ \mbox{is nonnegative with}\ \eta_0\neq 0,\\
c_0\in L^{\infty}(\Omega)\ \mbox{is nonnegative with}\ \sqrt{c_0}\in W^{1,2}(\Omega),\\  
{\mathbf{u}}_0\in L^2(\Omega),\ \nabla\cdot \mathbf{u}_0=0,
\end{array}
\right.
\end{equation}
in \cite{winkler2016global}, the author proves the following theorem of global existence of weak solutions.
\begin{theo}\label{teor2} Assume that the initial data $[n_0,c_0,{\mathbf{u}}_0]$ satisfies (\ref{initial1b}) and $\phi\in W^{2,\infty}({\Omega}).$ Then there exist
$\eta\in {L^{\infty}(0,T;L^{1}(\Omega))\cap {L^{\frac{5}{4}}([0,T],W^{1,\frac{5}{4}}}(\Omega))},$  $c\in {L^{\infty}(0,T;L^\infty(\Omega))\cap L^4([0,T];W^{1,4}(\Omega))},$ $\mathbf{u}\in L^2([0,T],\bf{V}),$
such that $[n,c,\bf{u}]$ is a weak solution of (\ref{KNS})-(\ref{initialdata}) in the sense of Definition \ref{weak}.
\end{theo}

 Notice that if $[\eta,c,\mathbf{u},\pi]$ is a classical solution of system (\ref{KNS})-(\ref{initialdata}), then we get the following basic property: 
\begin{equation}\label{mass01b}
\int_{\Omega} \eta(\cdot,t)\, d\x = \int_{\Omega} \eta_0 \, d\x. 
\end{equation}
Equality (\ref{mass01b}) is obtained integrating the equation (\ref{KNS})$_1$ and using the boundary conditions (\ref{initialdata})$_2$; and it means that the total mass of bacteria is conserved in time and equal to the initial mass.\\

In order to design the scheme, taking into account the property (\ref{mass01b}), we first consider the auxiliary variable $n:=\eta-\frac{1}{\vert\Omega\vert}\int_{\Omega} \eta=\eta-\frac{1}{\vert\Omega\vert}\int_{\Omega} \eta_0:=\eta -\alpha_0$. Also, 
in order to introduce the weak formulation for the $c$-equation (\ref{KNS})$_2$, we consider the auxiliary unknown $\boldsymbol{\sigma} = \nabla c.$ Thus, we obtain the following mixed variational form for the variables $\eta, c, \boldsymbol{\sigma},\mathbf{u}$ and $\pi$:

\begin{equation}\label{Chemoweak1}
\left\{
\begin{array}{lc}\vspace{0.1 cm}
(n_t, \bar{n})+ D_n(\nabla n,\nabla\bar{n})+( \mathbf{u} \cdot \nabla n,\bar{n}) =\chi((n+\alpha_0) \boldsymbol{\sigma}, \nabla \bar{n}),\quad \forall\bar{n}\in \widehat{H}^1(\Omega),&\\
(\boldsymbol{\sigma},\bar{\boldsymbol{\sigma}})+(c,\nabla \cdot \bar{\boldsymbol{\sigma}})=0, \quad \forall \bar{\boldsymbol{\sigma}} \in \mathbf{H}_{\sigma}^1(\Omega),&\\
\displaystyle (c_t ,\bar{c})-D_c(\nabla \cdot \boldsymbol{\sigma},\bar{c})+(\mathbf{u} \cdot\boldsymbol{\sigma}, \bar{c})=-\gamma((n+\alpha_0)c,\bar{c}), \quad \forall \bar{c} \in L^2(\Omega),&\\
\rho ( \mathbf{u}_t,\bar{\mathbf{u}})+ D_{\mathbf{u}} (\nabla \mathbf{u},\nabla\bar{\mathbf{u}})+ \rho((\mathbf{u}\cdot \nabla)\mathbf{u},\bar{\mathbf{u}}) - (\pi,\nabla\cdot\bar{\mathbf{u}})=((n+\alpha_0)\nabla\phi,\bar{\mathbf{u}}),\ \  \forall\bar{\mathbf{u}}\in \mathbf{H}_0^1(\Omega),& \\
(\bar{\pi},\nabla \cdot {\mathbf{u}}) =0,\quad \forall\bar{\pi}\in {L}_0^2(\Omega).&
\end{array}
\right.
\end{equation}
Taking the time derivative in the equation (\ref{Chemoweak1})$_2$ we have
\begin{eqnarray}\label{Chemoweak2}
(\boldsymbol{\sigma}_t, \bar{\boldsymbol{\sigma}})+(c_t, \nabla \cdot \bar{\boldsymbol{\sigma}})=0, \quad \forall \bar{\boldsymbol{\sigma}} \in \mathbf{H}_{\sigma}^1(\Omega).
\end{eqnarray}
Then, choosing $\bar{c} = \nabla \cdot\ \bar{\boldsymbol{\sigma}}$ in (\ref{Chemoweak1})$_3$, substituting it into (\ref{Chemoweak2}), and adding the term $D_c(\mbox{rot }{\boldsymbol{\sigma}},\mbox{rot }\bar{\boldsymbol{\sigma}})$ using that $\mbox{rot }{\boldsymbol{\sigma}}=\mbox{rot}(\nabla c)=0$, we get the following variational formulation:
\begin{equation}\label{Chemoweak3}
\left\{
\begin{array}{lc} \vspace{0.1 cm}
(n_t, \bar{n})+ D_n(\nabla n,\nabla\bar{n})+( \mathbf{u} \cdot \nabla n,\bar{n}) =\chi((n+\alpha_0) \boldsymbol{\sigma}, \nabla \bar{n}),&\\
(\boldsymbol{\sigma}_t,\bar{\boldsymbol{\sigma}})+D_c(\nabla \cdot \boldsymbol{\sigma},\nabla \cdot \bar{\boldsymbol{\sigma}})+D_c(\mbox{rot }{\boldsymbol{\sigma}},\mbox{rot }\bar{\boldsymbol{\sigma}})- (\mathbf{u}\cdot \boldsymbol{\sigma}, \nabla \cdot \bar{\boldsymbol{\sigma}}) =\gamma ((n+\alpha_0)c, \nabla \cdot \bar{\boldsymbol{\sigma}}),& \\
(c_t, \bar{c}) + D_c(\nabla c,\nabla \bar{c}) + (\mathbf{u}\cdot \nabla c,\bar{c})= -\gamma((n+\alpha_0)c, \bar{c}),&\\
\rho ( \mathbf{u}_t,\bar{\mathbf{u}})+ D_{\mathbf{u}} (\nabla \mathbf{u},\nabla\bar{\mathbf{u}})+ \rho((\mathbf{u}\cdot \nabla)\mathbf{u},\bar{\mathbf{u}}) - (\pi,\nabla\cdot\bar{\mathbf{u}})=((n+\alpha_0)\nabla\phi,\bar{\mathbf{u}}),& \\
(\bar{\pi},\nabla \cdot {\mathbf{u}}) =0,&
\end{array}
\right.
\end{equation}
for all $[\bar{n},\bar{\boldsymbol{\sigma}},\bar{c},\bar{\mathbf{u}},\bar{\pi}] \in  \widehat{H}^1(\Omega)\times \mathbf{H}_{\sigma}^1(\Omega) \times H^1(\Omega)\times \mathbf{H}_0^1(\Omega)\times {L}_0^2(\Omega)$. 
The integral equations (\ref{Chemoweak3}) define a variational form of  (\ref{KNS})-(\ref{initialdata}). Also, it is straightforward to check that if we have a smooth solution of (\ref{Chemoweak3}), then $\boldsymbol{\sigma} = \nabla c.$

\section{Numerical scheme}\label{NScheme}
In this section, we construct  a mass-conservative numerical approximation for the weak solutions of the chemotaxis-fluid system (\ref{KNS}) with initial and boundary data (\ref{initialdata}). The idea is to use finite element approximations in space and finite differences in time (considered for simplicity on a uniform partition of $[0,T]$ with time step
$\Delta t=T/N : (t_n = n\Delta t)_{n=0}^{n=N}$), combined with splitting ideas to decouple the computation of the fluid part from the chemotaxis one. Moreover, in order to deal with the velocity trilinear form and the nonlinear convective terms, we will use the skew-symmetric forms $A$ and $B$ given in (\ref{a6})-(\ref{a6aa}).  Concerning to the space discretization, we consider conformed finite element spaces:
$\mathcal{X}_n\times \mathcal{X}_c\times \mathcal{X}_{\boldsymbol{\sigma}} \times\mathcal{X}_{\textbf{u}}\times \mathcal{X}_\pi  \subset \widehat{H}^1(\Omega)\times {H}^1(\Omega)\times \mathbf{H}^1_{\sigma}(\Omega)\times \textbf{H}^1_0(\Omega)\times {L}^2_0(\Omega)$ corresponding to a family of shape-regular and quasi-uniform  triangulations of $\overline{\Omega}$, $\{\mathcal{T}_h\}_{h>0}$,  made up of triangles $K$, so that $\overline{\Omega}= \cup_{K\in \mathcal{T}_h} K$, where $h = \max_{K\in \mathcal{T}_h} h_K$, with $h_K$ being the diameter of $K$. 

\subsection{Hyphotesis and preliminary results}
We assume that $\mathcal{X}_{\textbf{u}}$ and $\mathcal{X}_\pi $ satisfy the following discrete {\it inf-sup} condition: There exists a constant $\beta>0$, independent of $h$, such that
\begin{equation}\label{LBB}
\sup_{\mathbf{v}\in \mathcal{X}_{\textbf{u}}\setminus\{0\}} \frac{-(w,\nabla \cdot \mathbf{v}) }{\|\mathbf{v}\|_{\mathcal{X}_{\textbf{u}}}} \geq \beta \|w\|_{\mathcal{X}_\pi}, \quad \forall w
\in \mathcal{X}_\pi.
\end{equation}
\begin{remark}{\bf (Some possibilities for the choice of the discrete spaces)} 
For the spaces $[\mathcal{X}_{\mathbf{u}}, \mathcal{X}_\pi]$, we can choose the Taylor-Hood approximation $\mathbb{P}_{r}\times \mathbb{P}_{r-1}$ (for $r\geq 2$) (\cite{GR,Sten}); or the approximation $\mathbb{P}_1-bubble \times \mathbb{P}_1$ (\cite{GR}) (for $r=1$). On the other hand, the spaces 
$[\mathcal{X}_{n}, \mathcal{X}_c,\mathcal{X}_{\boldsymbol{\sigma}}]$ are approximated by $\mathbb{P}_{r_1}\times\mathbb{P}_{r_2}\times\mathbb{P}_{r_3}$- continuous FE, with $r_i\geq 1$ ($i=1,2,3$).
\end{remark}

Then, we can define the well-known Stokes operator $(\mathbb{P}_{\mathbf{u}},\mathbb{P}_\pi):\mathbf{H}^1_0(\Omega)\times L^2_0(\Omega)\rightarrow \mathcal{X}_{\mathbf{u}}\times \mathcal{X}_\pi$ such that  $[\mathbb{P}_{\mathbf{u}}\mathbf{u},\mathbb{P}_\pi \pi]\in \mathcal{X}_{\mathbf{u}}\times \mathcal{X}_\pi$ satisfies
\begin{equation}\label{StokesOp}
\left\{
\begin{array}{lc} \vspace{0.1 cm}
D_{\mathbf{u}}(\nabla
(\mathbb{P}_{\mathbf{u}}{\mathbf{u}}  -{\mathbf{u}}), \nabla\bar{\mathbf{u}}) - (\mathbb{P}_\pi \pi - \pi, \nabla \cdot \bar{\mathbf{u}})=0,\quad \forall \bar{\mathbf{u}}\in \mathcal{X}_{\mathbf{u}},& \\
(\nabla\cdot
(\mathbb{P}_{\mathbf{u}}{\mathbf{u}}  -{\mathbf{u}}), \bar{\pi})=0, \quad \forall \bar{\pi}\in \mathcal{X}_\pi,&
\end{array}
\right.
\end{equation}
and the following approximation and stability properties hold (\cite{GV}):
\begin{equation}\label{StabStk1}
\| [\mathbf{u}-\mathbb{P}_ {\mathbf{u}}\mathbf{u},\pi-\mathbb{P}_{\pi} \pi]\|_{ H^1\times L^2} + \frac{1}{h}\| \mathbf{u}-\mathbb{P}_ {\mathbf{u}}\mathbf{u}\|_{L^2} \leq K h^r \|[\mathbf{u},\pi]\|_{H^{r+1}\times H^{r}},
\end{equation}
\begin{equation}\label{StabStk2}
\|[\mathbb{P}_ {\mathbf{u}}\mathbf{u},\mathbb{P}_{\pi} \pi]\|_{W^{1,6}\times L^6} \leq C \|[\mathbf{u},\pi]\|_{H^{2}\times H^1}.
\end{equation}
Moreover, we consider the following interpolation operators:  
$$
\mathbb{P}_n:\widehat{H}^1(\Omega)\rightarrow \mathcal{X}_n, \quad \mathbb{P}_c:H^1(\Omega)\rightarrow \mathcal{X}_c,\quad  \mathbb{P}_{\boldsymbol \sigma}:\mathbf{H}^1_{\sigma}(\Omega)\rightarrow \mathcal{X}_{\boldsymbol{\sigma}},
$$
such that for all $n\in \widehat{H}^1(\Omega)$, $c\in H^1(\Omega)$ and ${\boldsymbol{\sigma}}\in \mathbf{H}^1_{\sigma}(\Omega)$, $\mathbb{P}_n n \in \mathcal{X}_n$, $\mathbb{P}_c c\in \mathcal{X}_c$ and $\mathbb{P}_{\boldsymbol \sigma} {\boldsymbol \sigma}\in \mathcal{X}_{\boldsymbol{\sigma}}$ satisfy
\begin{equation}\label{Interp1New}
\left\{
\begin{array}{lc} \vspace{0.1 cm}
(\nabla
(\mathbb{P}_n n - n), \nabla\bar{n})=0,\quad \forall \bar{n}\in \mathcal{X}_n,& \\
\vspace{0.1 cm}
(\nabla
(\mathbb{P}_c c - c), \nabla\bar{c})  + (\mathbb{P}_c c - c, \bar{c})=0,\quad \forall \bar{c}\in \mathcal{X}_c,& \\
\vspace{0.1 cm}
(\nabla \cdot (\mathbb{P}_{\boldsymbol \sigma} {\boldsymbol\sigma} - {\boldsymbol\sigma}),\nabla \cdot \bar{\boldsymbol\sigma}) + (\mbox{rot}(\mathbb{P}_{\boldsymbol \sigma} {\boldsymbol\sigma}-{\boldsymbol\sigma}),\mbox{rot }\bar{\boldsymbol\sigma}) + (\mathbb{P}_{\boldsymbol \sigma} {\boldsymbol\sigma} - {\boldsymbol\sigma},\bar{\boldsymbol\sigma})=0,\ \ \forall \bar{\boldsymbol \sigma}\in \mathcal{X}_{\boldsymbol{\sigma}},&
\end{array}
\right.
\end{equation}
respectively. Observe that from Lax-Milgram Theorem, we have that the interpolation operators $\mathbb{P}_n$, $\mathbb{P}_c$ and $\mathbb{P}_{\boldsymbol \sigma}$ are well defined. Moreover, it is well known that the following interpolation errors hold:
\begin{equation}\label{aprox01} 
\left\{
\begin{array}{lc}
\| n-\mathbb{P}_n n\|_{L^2}+ h \|  n-\mathbb{P}_n n  \|_{H^1} \leq K h^{r_1+1} \|n\|_{H^{r_1+1}}, & \forall n \in H^{r_1+1}(\Omega),\\[.2cm]
\| c-\mathbb{P}_c c\|_{L^2}+ h \|  c-\mathbb{P}_c c  \|_{H^1} \leq K h^{r_2+1} \|c\|_{H^{r_2+1}}, & \forall c \in H^{r_2+1}(\Omega),\\[.2cm]
 \Vert {\boldsymbol\sigma}- \mathbb{P}_{\boldsymbol \sigma} {\boldsymbol\sigma}  \Vert_{L^2} + h \Vert {\boldsymbol\sigma} - \mathbb{P}_{\boldsymbol \sigma} {\boldsymbol\sigma} \Vert_{H^1} \leq Ch^{r_3 + 1} \Vert  {\boldsymbol\sigma}\Vert_{H^{r_3 +1}}, &  \forall  {\boldsymbol\sigma}\in \mathbf{H}^{r_3 +1}(\Omega).
\end{array}
\right.
\end{equation}
Also, the stability properties 
\begin{equation}\label{aprox01-aNN}
\| [\mathbb{P}_n n, \mathbb{P}_c c, \mathbb{P}_{\boldsymbol\sigma} {\boldsymbol{\sigma}}] \|_{H^{1}} \leq  \| [n,c, {\boldsymbol{\sigma}}] \|_{H^1},
\end{equation}
\begin{equation}\label{aprox01-a}
\| [\mathbb{P}_n n, \mathbb{P}_c c, \mathbb{P}_{\boldsymbol\sigma} {\boldsymbol{\sigma}}] \|_{W^{1,6}} \leq C \| [n,c, {\boldsymbol{\sigma}}] \|_{H^2},
\end{equation}
hold. Inequality (\ref{aprox01-aNN}) can be deduced from (\ref{Interp1New}), and (\ref{aprox01-a}) can be obtained from (\ref{aprox01}) using the inverse inequality 
$$\Vert [n_h,c_h,{\boldsymbol\sigma}_h]\Vert_{W^{1,6}}\leq Ch^{-p}\Vert [n_h,c_h,{\boldsymbol\sigma}_h]\Vert_{H^1} \ \ \mbox {for all} \ [n_h,c_h,{\boldsymbol\sigma}_h]\in \mathbb{P}_n\times \mathbb{P}_c\times \mathbb{P}_{\boldsymbol \sigma},$$ 
with $p=2/3$ (in the 2D case) and $p=1$ (in the 3D case), and comparing $\mathbb{P}_{n,c,{\boldsymbol \sigma}}$ with an average interpolation of Clement or Scott-Zhang type (which are stable in $W^{1,6}$-norm). Finally, we consider the following skew-symmetric trilinear forms which will be used in the formulation of the numerical scheme:
\begin{eqnarray}
B(\mathbf{v}_1,\mathbf{v}_2,\mathbf{v}_3)&=&\frac{1}{2}\Big[\Big( (\mathbf{v}_1\cdot \nabla)\mathbf{v}_2,\mathbf{v}_3 \Big) - \Big( (\mathbf{v}_1\cdot \nabla)\mathbf{v}_3,\mathbf{v}_2 \Big)\Big], \ \forall \mathbf{v}_1,\mathbf{v}_2,\mathbf{v}_3 \in \mathbf{H}^1(\Omega),\label{a6}\\
A(\mathbf{v},w_1,w_2)&=&\frac{1}{2}\Big[\Big( (\mathbf{v}\cdot \nabla)w_1,w_2 \Big) - \Big( (\mathbf{v}\cdot \nabla)w_2,w_1 \Big)\Big], \ \forall \mathbf{v} \in \mathbf{H}^1(\Omega), \ w_1,w_2\in H^1(\Omega). \label{a6aa}
\end{eqnarray}
It is easy to verify that
\begin{eqnarray}
B(\mathbf{v}_1,\mathbf{v}_2,\mathbf{v}_3)\!&=&\!\Big( (\mathbf{v}_1\cdot \nabla)\mathbf{v}_2,\mathbf{v}_3 \Big), \ \ \forall \mathbf{v}_1\in \mathbf{V} , \ \ \mathbf{v}_2,\mathbf{v}_3  \in  \mathbf{H}^1(\Omega), \label{aA1} \\
A(\mathbf{v}_1,w_1,w_2)\| &=&\!\Big( (\mathbf{v}_1\cdot \nabla)w_1,w_2 \Big), \ \ \forall \mathbf{v}_1\in \mathbf{V} , \ \ w_1,w_2  \in  {H}^1(\Omega),\label{aA2} \\
B(\mathbf{v}_1,\mathbf{v}_2,\mathbf{v}_2)&=&0,\ \ \forall \mathbf{v}_1,\mathbf{v}_2\in \mathbf{H}^1(\Omega), \label{a7} \\
A(\mathbf{v},w,w)&=&0,\ \forall w\in H^1(\Omega),\ \ \mathbf{v}\in \mathbf{H}^1(\Omega). \label{a8} 
\end{eqnarray}

\subsection{Definition of the scheme}
Taking into account  (\ref{Chemoweak3}), we consider the following first order in time, linear and semi-coupled scheme:\\

\textbf{Initialization:} Let $[n^0_h,c^0_h,\boldsymbol{\sigma}^0_h,\mathbf{u}^0_h]=[\mathbb{P}_n n_0, \mathbb{P}_c c_0, \mathbb{P}_{\boldsymbol\sigma} {\boldsymbol\sigma}_0, \mathbb{P}_{\mathbf{u}} \mathbf{u}_0] \in \mathcal{X}_n \times \mathcal{X}_c \times \mathcal{X}_{\boldsymbol\sigma} \times \mathcal{X}_{\mathbf{u}}$. \\

\textbf{Time step $m$:} Given the vector $[n^{m-1}_h,c^{m-1}_h,\boldsymbol{\sigma}^{m-1}_h,\mathbf{u}^{m-1}_h]\in \mathcal{X}_n \times \mathcal{X}_c \times \mathcal{X}_{\boldsymbol\sigma} \times \mathcal{X}_{\mathbf{u}}$, compute $[n^m_h,c^m_h,\boldsymbol{\sigma}^{m}_h,\mathbf{u}^m_h,\pi_h^m]\in\mathcal{X}_n \times \mathcal{X}_c \times \mathcal{X}_{\boldsymbol\sigma} \times \mathcal{X}_{\mathbf{u}}\times \mathcal{X}_\pi$ such that for each $[\bar{n},\bar{c},\bar{\boldsymbol{\sigma}},\bar{\mathbf{u}},\bar{\pi}] \in \mathcal{X}_n \times \mathcal{X}_c \times \mathcal{X}_{\boldsymbol\sigma} \times \mathcal{X}_{\mathbf{u}}\times \mathcal{X}_\pi$ it holds:
\begin{flalign}\label{scheme1} 
&a) \ \ (\delta_t n^m_h,\bar{n})+D_n (\nabla n^m_h,\nabla \bar{n})+A(\mathbf{u}^{m-1}_h, n^m_h,\bar{n}) - \chi((n^{m-1}_h+\alpha_0) {\boldsymbol{\sigma}}^{m-1}_h,\nabla \bar{n})= 0, \nonumber\\
&b) \ \ (\delta_t \boldsymbol{\sigma}^m_h,\bar{\boldsymbol{\sigma}})+D_c(\nabla \cdot \boldsymbol{\sigma}^m_h,\nabla \cdot \bar{\boldsymbol{\sigma}})+D_c (\mbox{rot }\boldsymbol{\sigma}^m_h,\mbox{rot }\bar{\boldsymbol{\sigma}})= (\mathbf{u}^{m-1}_h\cdot {\boldsymbol\sigma}^{m-1}_h +\gamma(n^{m-1}_h+\alpha_0)c^{m-1}_h, \nabla \cdot \bar{\boldsymbol{\sigma}}), & \nonumber\\
&c) \ \ (\delta_t c^m_h,\bar{c}) + D_c (\nabla c^m_h,\nabla \bar{c})+A(\mathbf{u}^{m-1}_h,c^m_h,\bar{c})=-\gamma((n^{m-1}_h+\alpha_0) c^{m-1}_h,\bar{c}), \\
&d) \ \ (\delta_t \mathbf{u}^m_h,\bar{\mathbf{u}})+B(\mathbf{u}^{m-1}_h,\mathbf{u}^m_h,\bar{\mathbf{u}})+\frac{D_{\mathbf{u}}}{\rho}(\nabla \mathbf{u}^m_h,\nabla \bar{\mathbf{u}})-\frac{1}{\rho}(\pi^m_h,\nabla \cdot \bar{\mathbf{u}}) = \frac{1}{\rho} ((n^{m-1}_h+\alpha_0) \nabla \phi,\bar{\mathbf{u}}), \nonumber \\
&e) \ \ (\bar{\pi},\nabla \cdot {\mathbf{u}_h^m}) =0, \nonumber
\end{flalign}
\noindent where, in general, we denote $\delta_t a^m_h=\frac{a^m_h-a^{m-1}_h}{\Delta t}$. \\

\begin{remark}
The skew-symmetric forms $A$ and $B$ verifying the  properties (\ref{aA1})-(\ref{a8}) will be important in order to get uniform estimates for the discrete solutions and to develop convergence analysis.
\end{remark}

\section{Well-posedness and mass-conservation}
In this section, we analyze the well-posedness of the scheme (\ref{scheme1})  and the mass-conservation property. With this aim, we define $\eta^m_h:=n^m_h+\alpha_0$. Then $\eta^m_h$ verifies
$$(\delta_t \eta^m_h,\bar{n})+D_n (\nabla \eta^m_h,\nabla \bar{n})+A(\mathbf{u}^{m-1}_h, \eta^m_h -\alpha_0,\bar{n}) - \chi(\eta^{m-1}_h{\boldsymbol{\sigma}}^{m-1}_h,\nabla \bar{n})= 0.$$ 

\begin{lemm}{\bf (Mass conservation)}\label{MCs}
	The discrete cell density $\eta^m_h$ satisfies the mass-conservation property (\ref{mass01b}).
\end{lemm}
\begin{proof}
	From the construction of scheme (\ref{scheme1}), for each $m\geq 0$ it holds that $\int_\Omega n^m_h=0$ (since $n^m_h\in \mathcal{X}_n\subset \widehat{H}^1(\Omega)$); from which we deduce that 
		$\eta^m_h$ satisfies 
		$$
		\int_\Omega \eta^m_h=\int_\Omega \eta^{m-1}_h=\cdot \cdot \cdot =\int_{\Omega} \eta^0_h=\alpha_0.
		$$
\end{proof}
Now, in the next result, we prove the well-posedness of the scheme (\ref{scheme1}).
\begin{theo}{\bf (Unconditional well-posedness)}\label{WPs}
The numerical scheme (\ref{scheme1}) is well-posed, that is, there exists a unique $[n^m_h,c^m_h,\boldsymbol{\sigma}^{m}_h,\mathbf{u}^m_h,\pi_h^m]\in\mathcal{X}_n \times \mathcal{X}_c \times \mathcal{X}_{\boldsymbol\sigma} \times \mathcal{X}_{\mathbf{u}}\times \mathcal{X}_\pi$ solution of the scheme (\ref{scheme1}).
\end{theo}
\begin{proof}
Taking into account that the scheme (\ref{scheme1}) is an algebraic linear system, it suffices to prove the uniqueness. For that, suppose that  there exist $[n^m_{h,1},c^m_{h,1},\boldsymbol{\sigma}^{m}_{h,1},\mathbf{u}^m_{h,1},\pi_{h,1}^m], [n^m_{h,2},c^m_{h,2},\boldsymbol{\sigma}^{m}_{h,2},\mathbf{u}^m_{h,2},\pi_{h,2}^m]\in\mathcal{X}_n \times \mathcal{X}_c \times \mathcal{X}_{\boldsymbol\sigma} \times \mathcal{X}_{\mathbf{u}}\times \mathcal{X}_\pi$ two possible solutions of the scheme (\ref{scheme1}).
Then defining $n^m_{h}=n^m_{h,1}-n^m_{h,2}$, $c^m_{h}=c^m_{h,1}- c^m_{h,2}$, $\boldsymbol{\sigma}^{m}_{h}=\boldsymbol{\sigma}^{m}_{h,1}-\boldsymbol{\sigma}^{m}_{h,2}$ $\mathbf{u}^m_{h}=\mathbf{u}^m_{h,1}-\mathbf{u}^m_{h,2}$ and $\pi_{h}^m=\pi_{h,1}^m - \pi_{h,2}^m$,
we have that $[n^m_h,c^m_h,\boldsymbol{\sigma}^{m}_h,\mathbf{u}^m_h,\pi_h^m]\in\mathcal{X}_n \times \mathcal{X}_c \times \mathcal{X}_{\boldsymbol\sigma} \times \mathcal{X}_{\mathbf{u}}\times \mathcal{X}_\pi$ satisfies
\begin{equation}
\left\{
\begin{array}
[c]{lll}%
\vspace{0.1 cm}
\displaystyle
(n^m_h,\bar{n})+\Delta tD_n (\nabla n^m_h,\nabla \bar{n})+\Delta t A(\mathbf{u}^{m-1}_h, n^m_h,\bar{n}) = 0,\\ \vspace{0.1 cm}
\displaystyle
( \boldsymbol{\sigma}^m_h,\bar{\boldsymbol{\sigma}})+\Delta t D_c(\nabla \cdot \boldsymbol{\sigma}^m_h,\nabla \cdot \bar{\boldsymbol{\sigma}})+\Delta t D_c (\mbox{rot }\boldsymbol{\sigma}^m_h,\mbox{rot }\bar{\boldsymbol{\sigma}})= 0,\\\vspace{0.1 cm}
\displaystyle
(c^m_h,\bar{c}) + \Delta t D_c (\nabla c^m_h,\nabla \bar{c})+\Delta t A(\mathbf{u}^{m-1}_h,c^m_h,\bar{c})=0,\\\vspace{0.1 cm} 
\rho (\mathbf{u}^m_h,\bar{\mathbf{u}})+\Delta t \rho B(\mathbf{u}^{m-1}_h,\mathbf{u}^m_h,\bar{\mathbf{u}})+\Delta t D_{\mathbf{u}}(\nabla \mathbf{u}^m_h,\nabla \bar{\mathbf{u}})-\Delta t(\pi^m_h,\nabla \cdot \bar{\mathbf{u}}) =0,\\
(\bar{\pi},\nabla \cdot {\mathbf{u}_h^m}) =0,
\end{array}
\right.  \label{modelf02clin}
\end{equation}
for all $[\bar{n},\bar{c},\bar{\boldsymbol{\sigma}},\bar{\mathbf{u}},\bar{\pi}] \in \mathcal{X}_n \times \mathcal{X}_c \times \mathcal{X}_{\boldsymbol\sigma} \times \mathcal{X}_{\mathbf{u}}\times \mathcal{X}_\pi$. Taking $[\bar{n},\bar{c},\bar{\boldsymbol \sigma},\bar{\mathbf{u}},\bar{\pi}]=[n^m_h,c^m_h,\boldsymbol{\sigma}^{m}_h,\mathbf{u}^m_h,\Delta t\pi_h^m]$ in (\ref{modelf02clin}), using the properties (\ref{a7})-(\ref{a8}), and adding, we obtain
\begin{equation*}
\left\{
\begin{array}
[c]{lll}%
\vspace{0.1 cm}
\displaystyle
\Vert n^m_h\Vert_{L^2}^2+\Delta tD_n \Vert \nabla n^m_h\Vert_{L^2}^2=0,\\ \vspace{0.1 cm}
\displaystyle
\Vert \boldsymbol{\sigma}^m_h\Vert_{L^2}^2+\Delta t D_c\Vert \nabla \cdot \boldsymbol{\sigma}^m_h\Vert_{L^2}^2+\Delta t D_c \Vert \mbox{rot }\boldsymbol{\sigma}^m_h\Vert_{L^2}^2= 0,\\\vspace{0.1 cm}
\displaystyle
\Vert c^m_h\Vert_{L^2}^2 + \Delta t D_c \Vert \nabla c^m_h\Vert_{L^2}^2=0,\\\vspace{0.1 cm} 
\rho \Vert \mathbf{u}^m_h\Vert_{L^2}^2+\Delta t D_{\mathbf{u}}\Vert \nabla \mathbf{u}^m_h\Vert_{L^2}^2 =0,
\end{array}
\right.  \label{uni2}
\end{equation*}
which implies that $[n^m_h,c^m_h,\boldsymbol{\sigma}^{m}_h,\mathbf{u}^m_h]=[0,0,\boldsymbol{0},\boldsymbol{0}]$. Finally, using that $\mathbf{u}^m_h=\boldsymbol{0}$ in (\ref{modelf02clin})$_4$, and using the discrete {\it inf-sup} condition (\ref{LBB}) we deduce that $\pi^m_h=0$. 
\end{proof}

\section{Uniform estimates and convergence}
In this section, we will obtain some uniform estimates for any solution of scheme (\ref{scheme1}) that will be used in the convergence analysis. Here, we focus on the two-dimensional case. Later, in subsection \ref{Sub3D}, we study the three-dimensional case. With this aim, we make the following inductive hypothesis:  there exists a positive constant $K>0$, independent of $m$, such that
\begin{equation}\label{IndHyp}
\Vert {\boldsymbol{\sigma}}^{m-1}_h\Vert_{H^1}\leq K, \qquad \forall m\geq 1.
\end{equation}
After the convergence analysis we verify the vality of (\ref{IndHyp}). It is worthwhile to remark that the use of induction hypothesis to deal with the convergence analysis of approximating schemes in nonlinear partial differential equations have been considered by several authors (see for instance, \cite{Cui, Douglas, Yang, Zhan-Zhu}). In \cite{Douglas} the authors analyze a numerical method for a model describing compressible miscible displacement in porous media, and consider an inductive hypothesis on the difference of pressures $\pi=p-p_h$ comming from the Darcy model, in norm $W^{1,\infty}.$ In the same spirit, by using inductive hypotheses, a superconvergence estimate of a Combined Mixed Finite Element and Discontinuous Galerkin Method for a Compressible Miscible Displacement Problem was obtained in \cite{Yang} (see also \cite{Cui}). In the context of the Keller-Segel system without fluid influence, in \cite{Zhan-Zhu}, an inductive hypothesis on  $\Vert  {\boldsymbol{\sigma}}^m_h\Vert_{W^{1,\infty}}$ is considered. Observe that the norm in (\ref{IndHyp}) is in $H^1(\Omega)$ instead of $W^{1,\infty}(\Omega).$\\

 Additionally, we will use the following discrete Gronwall lemma:
 \begin{lemm}\label{Diego2} (\cite[p. 369]{Hey})
 Assume that $\Delta t>0$ and $B,b^k, d^k, g^k, h^k\ge 0$ satisfy:
 	\begin{equation*}\label{e-Diego2-1}
 	d^{m+1}+ \Delta t\sum_{k=0}^{m} b^{k+1}  \le \Delta t \sum_{k=0}^{m} g^k \, d^k +\Delta t\sum_{k=0}^{m} h^k + B,  \quad \forall k \ge 0.
 	\end{equation*}
 	Then, it holds
 	\begin{equation*}\label{e-Diego2-2}
 	d^{m+1} + \Delta t \, \sum_{k=0}^{m} b^{k+1} \le \exp \left( \Delta t \, \sum_{k=0}^{m} g^k\right)
 	\, \left( \Delta t \, \sum_{k=0}^{m} h^k + B \right), \quad \forall k \ge 0.
 	\end{equation*}
 \end{lemm}
 
\subsection{Uniform estimates in finite time}
We can prove the following uniform estimates in finite time for the discrete cell and chemical variables, $n^m_h$ and $c^m_h$, in weak norms.
\begin{lemm}{\bf (Uniform estimates for $n_h^m$)}\label{uen} 
Assume the inductive hypothesis (\ref{IndHyp}). Then $n^m_h$ is bounded in $ l^{\infty}(L^2)\cap l^{2}(H^1)$.
\end{lemm}
\begin{proof}
Testing (\ref{scheme1})$_a$ by $\bar{n}=2\Delta t n^m_h$, using the equality $(a-b,2a)=\vert a\vert^2-\vert b\vert^2+\vert a-b\vert^2$, taking into account the property (\ref{a8}) and using the fact that $\int_{\Omega} n^m_h=0$ (since $n^m_h\in \mathcal{X}_n$), we have  
\begin{equation}\label{nh1}
\Vert n^m_h\Vert^2_{L^2}-\Vert n^{m-1}_h \Vert^2_{L^2}+ \Vert n^m_h - n^{m-1}_h\Vert^2_{L^2}+{2D_n\Delta t}\Vert n^m_h\Vert^2_{H^1}= 2 \Delta t \chi((n^{m-1}_h+\alpha_0) {\boldsymbol{\sigma}}^{m-1}_h,\nabla n^m_h).
\end{equation} 
Using the H\"older and Young inequalities and the 2D interpolation inequality (\ref{in2D}), we get
\begin{eqnarray}\label{nh2}
&2 \Delta t \chi((n^{m-1}_h+\alpha_0) {\boldsymbol{\sigma}}^{m-1}_h,\nabla n^m_h)&\!\!\!\!\leq  C\Delta t \chi ( \Vert n^{m-1}_h  \Vert_{L^2}^{1/2} \Vert n^{m-1}_h  \Vert_{H^1}^{1/2}+ \alpha_0)\Vert   {\boldsymbol{\sigma}}^{m-1}_h\Vert_{L^4} \Vert \nabla n^m_h\Vert_{L^2}  \nonumber\\
&&\hspace{-5 cm} \leq D_n{\Delta t}\Vert n^m_h\Vert^2_{H^1}+\frac{D_n\Delta t}{2}\Vert  n^{m-1}_h\Vert^2_{H^1}+ \frac{C\Delta t\chi^2}{D_n} \Vert   {\boldsymbol{\sigma}}^{m-1}_h\Vert_{L^4}^2 \Big(\frac{\chi^2}{D_n^2}\Vert n^{m-1}_h \Vert_{L^2}^2\Vert   {\boldsymbol{\sigma}}^{m-1}_h\Vert_{L^4}^2 + \alpha_0^2\Big).
\end{eqnarray}
Thus, taking into account the inductive hypothesis (\ref{IndHyp}), from (\ref{nh1})-(\ref{nh2}) we arrive at
\begin{equation}\label{nh3}
\Vert n^m_h\Vert^2_{L^2}-\|n^{m-1}_h\|^2_{L^2}
+
{D_n\Delta t}\Vert n^m_h\Vert^2_{H^1} - \frac{D_n\Delta t}{2}\Vert n^{m-1}_h\Vert^2_{H^1}
\leq
 C\frac{\chi^4}{D_n^3}\Delta t \Vert n^{m-1}_h \Vert_{L^2}^2+  C\frac{\chi^2}{D_n}\Delta t \alpha_0^2.
\end{equation}
 Then, summing in (\ref{nh3})  from $m=1$ to $m=r$, we arrive at
 \begin{eqnarray}\label{nh3-a}
& \Vert n^r_h\Vert^2_{L^2}&\!\!\!\!
 +
 \frac{D_n}{2}\Delta t\sum_{m=1}^r\Vert n^m_h\Vert^2_{H^1} + \frac{D_n}{2}\Delta t\Vert n^r_h\Vert^2_{H^1} \nonumber\\
 &&
 \leq \|n^0_h\|^2_{L^2}+  \frac{D_n\Delta t}{2}\Vert n^0_h\Vert^2_{H^1}+
 C\frac{\chi^4}{D_n^3}\Delta t\sum_{m=1}^{r} \Vert n^{m-1}_h \Vert_{L^2}^2+  C\frac{\chi^2}{D_n}\Delta t \sum_{m=1}^{r}\alpha_0^2.
 \end{eqnarray}
 Therefore, applying Lemma \ref{Diego2} to (\ref{nh3-a}) we deduce
$$
\Vert n^r_h\Vert^2_{L^2} +
{\Delta t}\sum_{m=1}^{r}\Vert n^m_h\Vert^2_{H^1}
\leq  C, \ \ \forall m\geq1, 
$$
where the constant $C>0$ depends on $(\chi, D_n,\alpha_0,T,n_0)$, but is independent of $(\Delta t, h)$ and $r$;  thus we conclude the proof.
\end{proof}

\begin{lemm}{\bf (Uniform estimates for $c_h^m$)}\label{uec}
	Assume the inductive hypothesis (\ref{IndHyp}). Then, $c^m_h$ is bounded in $ l^{\infty}(L^2)\cap l^{2}(H^1)$.
\end{lemm}
\begin{proof}
Testing (\ref{scheme1})$_c$ by $\bar{c}=2\Delta t c^m_h$, using the equality $(a-b,2a)=\vert a\vert^2-\vert b\vert^2+\vert a-b\vert^2$ and taking into account the property (\ref{a8}), we obtain
	\begin{equation}\label{ch1}
	\begin{array}{l}
	\Vert c^m_h\Vert^2_{L^2}-\Vert c^{m-1}_h\Vert^2_{L^2} + 	\Vert c^m_h- c^{m-1}_h\Vert^2_{L^2}+{2D_c\Delta t}\Vert \nabla c^m_h\Vert^2_{L^2}= -2\gamma\Delta t((n^{m-1}_h+\alpha_0) c^{m-1}_h,c^m_h).
	\end{array}
	\end{equation} 
Using the H\"older and Young inequalities,  we have
\begin{eqnarray}\label{ch2}
&2\gamma\Delta t((n^{m-1}_h+\alpha_0) c^{m-1}_h,c^m_h)&\!\!\! \leq 2 \gamma\Delta t\Vert n^{m-1}_h+\alpha_0\Vert_{L^4} \Vert c^{m-1}_h\Vert_{L^2} \Vert c^m_h\Vert_{L^4}\nonumber\\
&& \!\!\! \leq   C \gamma\Delta t \Vert n^{m-1}_h+\alpha_0\Vert_{L^4} \Vert c^{m-1}_h\Vert_{L^2}( \Vert \nabla c^m_h\Vert_{L^2} + \Vert c^m_h\Vert_{L^2})  \nonumber\\
&& \!\!\! \leq   D_c \Delta t \Vert \nabla c^m_h\Vert_{L^2}^2 + \frac{1}{2} \Vert c^m_h - c^{m-1}_h\Vert_{L^2}^2 + \Delta t \Vert c^{m-1}_h\Vert_{L^2}^2\nonumber\\
&& +  C\gamma^2\Delta t \left[\frac{1}{D_c}+ \Delta t +1\right]\Vert n^{m-1}_h+\alpha_0\Vert_{L^4}^2 \Vert c^{m-1}_h\Vert_{L^2}^2. 
\end{eqnarray}
Therefore, from (\ref{ch1}) and (\ref{ch2}), we arrive at
\begin{eqnarray}\label{ch3}
	&\Vert c^m_h\Vert^2_{L^2}&\!\!\!\!-\Vert c^{m-1}_h\Vert^2_{L^2} + \frac{1}{2}	\Vert c^m_h- c^{m-1}_h\Vert^2_{L^2}+{D_c\Delta t}\Vert \nabla c^m_h\Vert^2_{L^2}\nonumber\\
&&\!\!\!\!	\leq
C\gamma^2\Delta t \left[\frac{1}{D_c}+ \Delta t +1 \right]\Vert n^{m-1}_h+\alpha_0\Vert_{L^4}^2 \Vert c^{m-1}_h\Vert_{L^2}^2 + \Delta t \Vert c^{m-1}_h\Vert_{L^2}^2.
\end{eqnarray}
Then, summing in (\ref{ch3}) from $m=1$ to $m=r$, applying Lemma \ref{Diego2} and taking into account that $n^m_h$ is bounded in $l^2(H^1)$ (see Lemma \ref{uen}), we arrive at 
$$
\Vert c^r_h\Vert^2_{L^2} +
{\Delta t}\sum_{m=1}^{r}\Vert \nabla  c^m_h\Vert^2_{L^2}
\leq  C, \ \ \forall m\geq1, 
$$
the constant $C>0$ depends on $(\chi, D_c,D_n,\gamma,\alpha_0,T,n_0,c_0)$, but is independent of $(\Delta t, h)$ and $r$.
\end{proof}

\subsection{Error estimates in weak norms}\label{ESWN}
The aim of this subsection is to obtain optimal error estimates for any solution $[n^m_h, c^m_h,{\boldsymbol\sigma}^m_h,\mathbf{u}^m_h,\pi^m_h]$ of the scheme (\ref{scheme1}), with respect to a su\-ffi\-cien\-tly regular solution $[n, c,{\boldsymbol\sigma},\mathbf{u},\pi]$ of  (\ref{Chemoweak3}) in the two dimensional case. We start by introducing the following notations for the errors at $t=t_{m}$:  $e_n^m=n^m-n^m_h$, $e_c^m=c^m-c^m_h$, $e_{\boldsymbol\sigma}^m={\boldsymbol\sigma}^m-{\boldsymbol\sigma}^m_h$, $e_{\mathbf{u}}^m=\mathbf{u}^m-\mathbf{u}^m_h$ and $e_\pi^m=\pi^m-\pi^m_h$, where $w^{m}$ denote, in general, the value of $w$ at time $t_{m}$. Then, subtracting the scheme (\ref{scheme1}) to (\ref{Chemoweak3}) at $t=t_m$, we obtain that $[e_n^m,e^m_c,e_{{\boldsymbol \sigma}}^m,e^m_{\mathbf{u}},e^m_\pi]$ satisfies
\begin{eqnarray}\label{errn}
&(\delta_t e_n^m,&\!\!\!\!\bar{n}) + D_n(\nabla e_n^m, \nabla \bar{n})+A(\mathbf{u}^m-\mathbf{u}^{m-1},n^m,\bar{n})  +A(\mathbf{u}^{m-1}_h,e^m_n,\bar{n}) + A(e^{m-1}_{\mathbf{u}},n^m,\bar{n})=(\omega_n^m,\bar{n})\nonumber\\
&&\!\!\!\!\!\!\!\!\!\!\! +\chi((n^{m-1}\!+\alpha_0)({\boldsymbol{\sigma}}^m- {\boldsymbol{\sigma}}^{m-1})+ (n^m-n^{m-1}){\boldsymbol{\sigma}}^m+(n^{m-1}\!+\alpha_0)e^{m-1}_{\boldsymbol{\sigma}}+ e^{m-1}_n {\boldsymbol{\sigma}}^{m-1}_h,\nabla \bar{n}),
\end{eqnarray}
\begin{eqnarray}\label{errs}
&\left(\delta_t e_{\boldsymbol\sigma}^m,\bar{\boldsymbol \sigma}\right)&\!\!\!\!+ D_c(\nabla \cdot e_{\boldsymbol\sigma}^m, \nabla \cdot\bar{\boldsymbol \sigma}) + D_c(\mbox{rot }e_{\boldsymbol\sigma}^m, \mbox{rot }\bar{\boldsymbol \sigma})  = (\omega_{\boldsymbol\sigma}^m,\bar{\boldsymbol \sigma})+((\mathbf{u}^m-\mathbf{u}^{m-1})\cdot {\boldsymbol\sigma}^m, \nabla \cdot\bar{\boldsymbol \sigma}) \nonumber\\
&&\!\!\!\!\!\! + (\mathbf{u}^{m-1}({\boldsymbol{\sigma}}^m-{\boldsymbol{\sigma}}^{m-1})+ e^{m-1}_{\mathbf{u}}{\boldsymbol{\sigma}}^{m-1}+\mathbf{u}^{m-1}_h e^{m-1}_{\boldsymbol{\sigma}}+ \gamma(n^{m}\!-n^{m-1})c^m,\nabla \cdot\bar{\boldsymbol \sigma})\nonumber\\
&&\!\!\!\!\!\! + \gamma((n^{m-1}\!+\alpha_0)(c^m \!- c^{m-1}) + e^{m-1}_n c^{m-1} + (n^{m-1}_h\!+\alpha_0)e^{m-1}_c,\nabla \cdot\bar{\boldsymbol \sigma}),
\end{eqnarray}
\begin{eqnarray}\label{errc}
&\left(\delta_t e_c^m,\bar{c}\right) &\!\!\!\!+ D_c(\nabla e_c^m, \nabla \bar{c})+A(\mathbf{u}^m-\mathbf{u}^{m-1},c^m,\bar{c})  +A(\mathbf{u}^{m-1}_h,e^m_c,\bar{c}) + A(e^{m-1}_{\mathbf{u}},c^m,\bar{c})=(\omega_c^m,\bar{c})\nonumber\\
&&\hspace{-0.8 cm}  - \gamma((n^{m}\!-n^{m-1})c^m + (n^{m-1}\!+\alpha_0)(c^m \!- c^{m-1}) + e^{m-1}_n c^{m-1} + (n^{m-1}_h\!+\alpha_0)e^{m-1}_c, \bar{c}),
\end{eqnarray}
\begin{eqnarray}\label{erru}
&\left(\delta_t e_{\mathbf{u}}^m,\bar{\mathbf{u}}\right) &\!\!\!\!+ \frac{D_{\mathbf{u}}}{\rho}(\nabla e_{\mathbf{u}}^m, \nabla \bar{\mathbf{u}})= (\omega_{\mathbf{u}}^m,\bar{\mathbf{u}}) - B(\mathbf{u}^m-\mathbf{u}^{m-1},{\mathbf{u}}^m,\bar{\mathbf{u}})   -B(e^{m-1}_{\mathbf{u}},{\mathbf{u}}^m,\bar{\mathbf{u}})\nonumber\\
&&\hspace{-0.8 cm} -B(\mathbf{u}^{m-1}_h,e^m_{\mathbf{u}},\bar{\mathbf{u}}) +\frac{1}{\rho}(e^m_\pi,\nabla \cdot \bar{\mathbf{u}}) + \frac{1}{\rho}((n^m-n^{m-1})\nabla \phi,\bar{\mathbf{u}})+ \frac{1}{\rho}(e^{m-1}_n \nabla \phi,\bar{\mathbf{u}}),
\end{eqnarray}
\begin{equation}\label{errpi}
(\bar{\pi},\nabla \cdot e^m_{\mathbf{u}})=0,
\end{equation}
for all $[\bar{n},\bar{c},\bar{\boldsymbol \sigma},\bar{\mathbf{u}},\bar{\pi}]\in \mathcal{X}_n\times\mathcal{X}_c\times \mathcal{X}_{\boldsymbol \sigma}\times \mathcal{X}_{\mathbf{u}}\times \mathcal{X}_\pi$, where $\omega_n^m,\omega_c^m,\omega^m_{\boldsymbol{\sigma}},\omega^m_{\mathbf{u}}$ are the consistency errors associated to the scheme (\ref{scheme1}), that is, $\omega_n^m=\delta_t n^m - (n^m)_t$ and so on. \\

{\underline{{\it 1. Error estimate for the cell density} $n$}}\\

Considering the interpolation operators $\mathbb{P}_n, \mathbb{P}_{\boldsymbol{\sigma}},\mathbb{P}_{\mathbf{u}}$ defined in (\ref{Interp1New}) and (\ref{StokesOp}), we decompose the total errors $e_n^m, e_{\boldsymbol{\sigma}}^m,e_{\mathbf{u}}^m$  as follows:
\begin{equation}\label{u1a}
e_n^m=(n^m -\mathbb{P}_n n^m) + ( \mathbb{P}_n n^m - n^m_h )=\theta^m_n+\xi^m_n,
\end{equation}
\begin{equation}\label{u1asig}
e_{\boldsymbol{\sigma}}^m=({\boldsymbol{\sigma}}^m -\mathbb{P}_{\boldsymbol{\sigma}} {\boldsymbol{\sigma}}^m) + ( \mathbb{P}_{\boldsymbol{\sigma}} {\boldsymbol{\sigma}}^m - {\boldsymbol{\sigma}}^m_h )=\theta^m_{\boldsymbol{\sigma}}+\xi^m_{\boldsymbol{\sigma}},
\end{equation}
\begin{equation}\label{u1au}
e_{\mathbf{u}}^m=(\mathbf{u}^m -\mathbb{P}_{\mathbf{u}} \mathbf{u}^m) + ( \mathbb{P}_{\mathbf{u}} \mathbf{u}^m - \mathbf{u}^m_h )=\theta^m_{\mathbf{u}}+\xi^m_{\mathbf{u}}.
\end{equation}
Then, taking into account (\ref{Interp1New})$_1$, from (\ref{errn}) and (\ref{u1a})-(\ref{u1au}) we have
\begin{eqnarray}\label{errn-int}
&(\delta_t &\!\!\!\!\! \xi_n^m,\bar{n}) + D_n(\nabla \xi_n^m, \nabla \bar{n}) +A(\mathbf{u}^{m-1}_h,\xi^m_n,\bar{n}) =(\omega_n^m,\bar{n})- \left(\delta_t \theta_n^m,\bar{n}\right)-A(\mathbf{u}^{m-1}_h,\theta^m_n,\bar{n}) \nonumber\\
&&-A(\mathbf{u}^m-\mathbf{u}^{m-1},n^m,\bar{n}) - A((\xi^{m-1}_{\mathbf{u}}+\theta^{m-1}_{\mathbf{u}}),n^m,\bar{n})  +\chi((n^{m-1}\!+\alpha_0)({\boldsymbol{\sigma}}^m - {\boldsymbol{\sigma}}^{m-1}),\nabla \bar{n})\nonumber\\
&& +\chi( (n^m-n^{m-1}){\boldsymbol{\sigma}}^m+(n^{m-1}\!+\alpha_0)(\xi^{m-1}_{\boldsymbol{\sigma}}+\theta^{m-1}_{\boldsymbol{\sigma}})+ (\xi^{m-1}_n+\theta^{m-1}_n) {\boldsymbol{\sigma}}^{m-1}_h,\nabla \bar{n}).
\end{eqnarray}
Taking $\bar{n}=\xi_n^m$ in (\ref{errn-int}), using (\ref{a8}) and taking into account that $\int_\Omega \xi_n^m=0$ (since $\xi_n^m\in \mathcal{X}_n$), we get
\begin{eqnarray}\label{errn-int2}
&\displaystyle\frac{1}{2}&\!\!\!\!\!\delta_t  \Vert \xi_n^m\Vert_{L^2}^2 + \frac{\Delta t}{2} \Vert \delta_t \xi_n^m\Vert_{L^2}^2 + D_n\Vert\xi_n^m\Vert_{H^1}^2 = (\omega_n^m,\xi_n^m)- \left(\delta_t \theta_n^m,\xi_n^m\right)-A(\mathbf{u}^{m-1}_h,\theta^m_n,\xi_n^m) \nonumber\\
&&\!\!\!\!\!\!\!\!\!-A(\mathbf{u}^m-\mathbf{u}^{m-1},n^m,\xi_n^m) - A((\xi^{m-1}_{\mathbf{u}}+\theta^{m-1}_{\mathbf{u}}),n^m,\xi_n^m)  +\chi((n^{m-1}\!+\alpha_0)({\boldsymbol{\sigma}}^m - {\boldsymbol{\sigma}}^{m-1}),\nabla \xi_n^m)\nonumber\\
&&\!\!\!\!\!\!\!\!\! +\chi( (n^m-n^{m-1}){\boldsymbol{\sigma}}^m+(n^{m-1}\!+\alpha_0)(\xi^{m-1}_{\boldsymbol{\sigma}}+\theta^{m-1}_{\boldsymbol{\sigma}})+ (\xi^{m-1}_n+\theta^{m-1}_n) {\boldsymbol{\sigma}}^{m-1}_h,\nabla \xi_n^m)\nonumber\\
&&\!\!\!\!\!\!\!\!\! =\sum_{k=1}^9 I_k. 
\end{eqnarray}
Then, using the H\"older and Young inequalities, (\ref{StabStk1})-(\ref{StabStk2}), (\ref{aprox01}) and (\ref{aprox01-a}),  we control the terms on the right hand side of (\ref{errn-int2}) as follows:
\begin{equation}\label{Ea1a}
I_1\leq \displaystyle\frac{D_n}{12} \Vert \xi_n^m\Vert_{H^1}^2 + \frac{C}{D_n}\Vert \omega_n^m\Vert_{(H^1)'}^2\leq \displaystyle\frac{D_n}{12}\Vert\xi_n^m\Vert_{H^1}^2+\frac{C\Delta t}{D_n} \int_{t_{m-1}}^{t_m}\Vert n_{tt}(t)\Vert_{(H^1)'}^2 dt,
\end{equation}
\begin{eqnarray}\label{Ea1b}
&I_2&\!\!\!\leq  \Vert \xi_n^m\Vert_{L^2} \Vert(\mathcal{I} - \mathbb{P}_n) \delta_t n^m \Vert_{L^2} \leq \displaystyle\frac{D_n}{12} \Vert \xi_n^m\Vert_{L^2}^2+\frac{C h^{2(r_1+1)}}{D_n}\Vert \delta_t n^m\Vert_{H^{r_1+1}}^2\nonumber\\
&&\!\!\!\leq \displaystyle\frac{D_n}{12} \Vert \xi_n^m\Vert_{L^2}^2+\displaystyle\frac{C h^{2(r_1+1)}}{D_n \Delta t }\int_{t_{m-1}}^{t_m}\Vert n_t \Vert_{H^{r_1+1}}^2  dt,
\end{eqnarray}
\begin{eqnarray}\label{Ea1c}
&I_3&\!\!\! = A(\xi_{\mathbf{u}}^{m-1},\theta^m_n,\xi^m_n) -A(\mathbb{P}_{\mathbf{u}} \mathbf{u}^{m-1},\theta^m_n,\xi^m_n)  \nonumber\\
&&\!\!\! \leq  \Vert \xi_{\mathbf{u}}^{m-1}\Vert_{L^2}   \Vert \theta_n^m\Vert_{L^\infty\cap W^{1,3}} \Vert \xi_n^m\Vert_{H^1} + \Vert \mathbb{P}_{\mathbf{u}} \mathbf{u}^{m-1}\Vert_{L^\infty\cap W^{1,3}}   \Vert \theta_n^m\Vert_{L^2} \Vert \xi_n^m\Vert_{H^1} \nonumber\\
&&\!\!\! \leq \displaystyle\frac{D_n}{12} \Vert \xi_n^m\Vert_{H^1}^2+\frac{C}{D_n} \Vert n^m\Vert_{H^2}^2 \Vert \xi_{\mathbf{u}}^{m-1}\Vert_{L^2}^2    + \displaystyle\frac{C}{D_n} h^{2(r_1 +1)} \Vert [\mathbf{u}^{m-1},\pi^{m-1}]\Vert_{H^2\times H^1}^2   \Vert n^m\Vert_{H^{r_1 + 1}}^2,
\end{eqnarray}
\begin{eqnarray}\label{Ea1d}
&I_4 + I_5&\!\!\! \leq (\Vert \mathbf{u}^m-\mathbf{u}^{m-1}\Vert_{L^2}+\Vert \xi_{\mathbf{u}}^{m-1}\Vert_{L^2}+\Vert \theta_{\mathbf{u}}^{m-1}\Vert_{L^2}  )  \Vert n^m\Vert_{L^\infty\cap W^{1,3}} \Vert \xi_n^m\Vert_{H^1}  \nonumber\\
&&\!\!\! \leq \displaystyle\frac{D_n}{12} \Vert \xi_n^m\Vert_{H^1}^2+ \frac{C}{D_n} \Vert \mathbf{u}^m-\mathbf{u}^{m-1}\Vert_{L^2}^2  \Vert n^m\Vert_{H^2}^2\nonumber\\
&&+\displaystyle\frac{C}{D_n} (h^{2(r+1)} \|[\mathbf{u}^{m-1},\pi^{m-1}]\|_{H^{r+1}\times H^{r}}^2+\Vert \xi_{\mathbf{u}}^{m-1}\Vert_{L^2}^2) \Vert n^m\Vert_{H^2}^2,
\end{eqnarray}
\begin{eqnarray}\label{Ea1e}
&\displaystyle\sum_{k=6}^8 I_k&\!\!\! \leq \chi\Vert [{\boldsymbol{\sigma}}^m\! -\! {\boldsymbol{\sigma}}^{m-1}\!, n^m \!- \!n^{m-1}\!, \xi^{m-1}_{\boldsymbol{\sigma}}, \theta^{m-1}_{\boldsymbol{\sigma}}]\Vert_{L^2}  \Vert[n^{m-1}+\alpha_0,{\boldsymbol{\sigma}}^m]\Vert_{L^\infty} \Vert \xi_n^m\Vert_{H^1}  \nonumber\\
&&\!\!\! \leq \displaystyle\frac{D_n}{12} \Vert \xi_n^m\Vert_{H^1}^2 +\frac{C\chi^2}{D_n} (\Vert [{\boldsymbol{\sigma}}^m\! -\! {\boldsymbol{\sigma}}^{m-1}\!, n^m \!- \!n^{m-1}]\Vert_{L^2}^2 +\Vert \xi^{m-1}_{\boldsymbol{\sigma}}\Vert_{L^2}^2) \Vert[n^{m-1}+\alpha_0,{\boldsymbol{\sigma}}^m]\Vert_{L^\infty}^2\nonumber\\
&&+\displaystyle\frac{C\chi^2}{D_n}h^{2(r_3+1)} \|{\boldsymbol{\sigma}}^{m-1}\|_{ H^{r_3+1}}^2 \Vert[n^{m-1}\!+\alpha_0,{\boldsymbol{\sigma}}^m]\Vert_{L^\infty}^2,
\end{eqnarray}
\begin{eqnarray}\label{Ea1e-new}
& I_9&\!\!\! = \chi (\xi^{m-1}_n {\boldsymbol{\sigma}}^{m-1}_h,\nabla \xi_n^m) - \chi(\theta^{m-1}_n \xi_{\boldsymbol{\sigma}}^{m-1},\nabla \xi_n^m)+ \chi(\theta^{m-1}_n \mathbb{P}_{\boldsymbol{\sigma}}{\boldsymbol{\sigma}}^{m-1},\nabla \xi_n^m) \nonumber\\
&&\!\!\! \leq \chi C(\Vert \xi^{m-1}_{n}\Vert_{L^2}^{1/2}\Vert \xi^{m-1}_{n}\Vert_{H^1}^{1/2} \Vert{\boldsymbol{\sigma}}^{m-1}_h\Vert_{L^4}+ \Vert \xi_{\boldsymbol\sigma}^{m-1}\Vert_{L^2}   \Vert \theta_n^{m-1}\Vert_{L^\infty}  + \Vert \mathbb{P}_{\boldsymbol{\sigma}} \boldsymbol{\sigma}^{m-1}\Vert_{L^\infty}  \Vert \theta_n^{m-1}\Vert_{L^2}) \Vert \xi_n^m\Vert_{H^1}  \nonumber\\
&&\!\!\! \leq \displaystyle\frac{D_n}{12} \Vert \xi_n^m\Vert_{H^1}^2 +  \displaystyle\frac{D_n}{8} \Vert \xi^{m-1}_{n}\Vert_{H^1}^{2}+ \frac{C\chi^4}{D_n^3} \Vert \xi^{m-1}_{n}\Vert_{L^2}^{2}  \nonumber\\
&&+\frac{C\chi^2}{D_n} \Vert n^{m-1}\Vert_{H^2}^2 \Vert \xi_{\boldsymbol{\sigma}}^{m-1}\Vert_{L^2}^2    + \displaystyle\frac{C\chi^2 }{D_n} h^{2(r_1 +1)} \Vert {\boldsymbol{\sigma}}^{m-1}\Vert_{H^2}^2   \Vert n^{m-1}\Vert_{H^{r_1 + 1}}^2,
\end{eqnarray}
where, in the first and second inequalities of (\ref{Ea1e-new}), the 2D interpolation inequality (\ref{in2D}) and the inductive hypothesis (\ref{IndHyp}) were used. Therefore, from (\ref{errn-int2})-(\ref{Ea1e-new}), we arrive at
\begin{eqnarray}\label{errnfin}
&\displaystyle\frac{1}{2}&\!\!\!\!\!\delta_t  \Vert \xi_n^m\Vert_{L^2}^2 + \frac{\Delta t}{2} \Vert \delta_t \xi_n^m\Vert_{L^2}^2 + \frac{D_n}{2}\Vert\xi_n^m\Vert_{H^1}^2 -  \displaystyle\frac{D_n}{8} \Vert \xi^{m-1}_{n}\Vert_{H^1}^{2} \leq   \displaystyle\frac{C h^{2(r_1+1)}}{D_n \Delta t }\int_{t_{m-1}}^{t_m}\Vert n_t \Vert_{H^{r_1+1}}^2  dt \nonumber\\
&&\!\!\!\!\!\!\!\!\!+\frac{C\Delta t}{D_n} \int_{t_{m-1}}^{t_m}\Vert n_{tt}(t)\Vert_{(H^1)'}^2 dt+\frac{C}{D_n} \Vert n^m\Vert_{H^2}^2 \Vert \xi_{\mathbf{u}}^{m-1}\Vert_{L^2}^2  + \displaystyle\frac{C}{D_n} h^{2(r_1 +1)} \Vert [\mathbf{u}^{m-1},\pi^{m-1}]\Vert_{H^2\times H^1}^2   \Vert n^m\Vert_{H^{r_1 + 1}}^2\nonumber\\
&&\!\!\!\!\!\!\!\!\!+ \frac{C}{D_n} \Vert \mathbf{u}^m-\mathbf{u}^{m-1}\Vert_{L^2}^2  \Vert n^m\Vert_{H^2}^2+\displaystyle\frac{C}{D_n} h^{2(r+1)} \|[\mathbf{u}^{m-1},\pi^{m-1}]\|_{H^{r+1}\times H^{r}}^2 \Vert n^m\Vert_{H^2}^2\nonumber\\
&&\!\!\!\!\!\!\!\!\! +\frac{C\chi^2}{D_n} (\Vert [{\boldsymbol{\sigma}}^m\! -\! {\boldsymbol{\sigma}}^{m-1}\!, n^m \!- \!n^{m-1}]\Vert_{L^2}^2+ \Vert \xi^{m-1}_{\boldsymbol{\sigma}}\Vert_{L^2}^2+h^{2(r_3+1)} \|{\boldsymbol{\sigma}}^{m-1}\|_{ H^{r_3+1}}^2 )  \Vert[n^{m-1}+\alpha_0,{\boldsymbol{\sigma}}^m]\Vert_{L^\infty}^2\nonumber\\
&&\!\!\!\!\!\!\!\!\!+ \frac{C\chi^4}{D_n^3} \Vert \xi^{m-1}_{n}\Vert_{L^2}^{2} +\frac{C\chi^2}{D_n} \Vert n^{m-1}\Vert_{H^2}^2 \Vert \xi_{\boldsymbol{\sigma}}^{m-1}\Vert_{L^2}^2    + \displaystyle\frac{C\chi^2 }{D_n} h^{2(r_1 +1)} \Vert {\boldsymbol{\sigma}}^{m-1}\Vert_{H^2}^2   \Vert n^{m-1}\Vert_{H^{r_1 + 1}}^2.
\end{eqnarray}
\vspace{0.5 cm}

{\underline{{\it 2. Error estimate for the chemical concentration} $c$}}\\

Considering the interpolation operator $\mathbb{P}_c$ defined in (\ref{Interp1New})$_2$, we decompose the total error $e_c^m$  as follows:
\begin{equation}\label{u1ac}
e_c^m=(c^m -\mathbb{P}_c c^m) + ( \mathbb{P}_c c^m - c^m_h )=\theta^m_c+\xi^m_c.
\end{equation}
Then, taking into account (\ref{Interp1New})$_2$, from (\ref{errc}), (\ref{u1a}), (\ref{u1au}) and (\ref{u1ac}), we have
\begin{eqnarray}\label{errc-int}
&(\delta_t &\!\!\!\!\xi_c^m,\bar{c}) + D_c(\nabla \xi_c^m, \nabla \bar{c})  +A(\mathbf{u}^{m-1}_h,\xi^m_c,\bar{c}) =(\omega_c^m,\bar{c})- (\delta_t \theta^m_c,\bar{c})-A(\mathbf{u}^m-\mathbf{u}^{m-1},c^m,\bar{c})\nonumber\\
&&\hspace{-0.3 cm}-A(\mathbf{u}^{m-1}_h,\theta^m_c,\bar{c})-A((\xi^{m-1}_{\mathbf{u}}+\theta^{m-1}_{\mathbf{u}}),c^m,\bar{c})   - \gamma((n^{m}\!-n^{m-1})c^m + (n^{m-1}\!+\alpha_0)(c^m \!- c^{m-1}), \bar{c})\nonumber\\
&&\hspace{-0.3 cm}- \gamma( (\xi^{m-1}_n+\theta^{m-1}_n) c^{m-1} + (n^{m-1}_h\!+\alpha_0)(\xi^{m-1}_c+\theta^{m-1}_c), \bar{c})+ D_c(\theta^m_c,\bar{c}).
\end{eqnarray}
Taking $\bar{c}=\xi_c^m$ in (\ref{errc-int}) and using (\ref{a8}), we get
\begin{eqnarray}\label{errc-int2}
&\displaystyle\frac{1}{2}&\!\!\!\!\!\delta_t  \Vert \xi_c^m\Vert_{L^2}^2 + \frac{\Delta t}{2} \Vert \delta_t \xi_c^m\Vert_{L^2}^2 + D_c\Vert \nabla \xi_c^m\Vert_{L^2}^2   =(\omega_c^m,\xi_c^m)- (\delta_t \theta^m_c,\xi_c^m)-A(\mathbf{u}^m-\mathbf{u}^{m-1},c^m,\xi_c^m)\nonumber\\
&&\hspace{-0.4 cm}-A(\mathbf{u}^{m-1}_h\!,\theta^m_c,\xi_c^m)-A((\xi^{m-1}_{\mathbf{u}}\!\!+\theta^{m-1}_{\mathbf{u}}),c^m,\xi_c^m)   - \gamma((n^{m}\!-n^{m-1})c^m + (n^{m-1}\!+\alpha_0)(c^m \!- c^{m-1}), \xi_c^m)\nonumber\\
&&\hspace{-0.4 cm}- \gamma( (\xi^{m-1}_n\!+\theta^{m-1}_n) c^{m-1}\! + (n^{m-1}_h\!+\alpha_0)(\xi^{m-1}_c\!+\theta^{m-1}_c), \xi_c^m)+ D_c(\theta^m_c,\xi_c^m)=\sum_{k=1}^{10} J_k.
\end{eqnarray}
Then, using the H\"older and Young inequalities, (\ref{StabStk1})-(\ref{StabStk2}), (\ref{aprox01}) and (\ref{aprox01-a}), we control the terms on the right hand side of (\ref{errc-int2}) as follows:
\begin{eqnarray}\label{Ea1aC}
&J_1&\!\!\!\leq (\Vert \xi_c^m\Vert_{L^2}+ \Vert \nabla \xi_c^m\Vert_{L^2})\Vert \omega_c^m\Vert_{(H^1)'}\leq  (\Delta t \Vert \delta_t \xi_c^m\Vert_{L^2}+ \Vert  \xi_c^{m-1}\Vert_{L^2}+ \Vert \nabla \xi_c^m\Vert_{L^2})\Vert \omega_c^m\Vert_{(H^1)'}  \nonumber\\
&&\!\!\!\leq \displaystyle\frac{D_c}{8} \Vert \nabla \xi_c^m\Vert_{L^2}^2 + \frac{\Delta t}{24}\Vert \delta_t \xi_c^m\Vert_{L^2}^2+ \frac{1}{2}\Vert  \xi_c^{m-1}\Vert_{L^2}^2 + \Big(\frac{C}{D_c} + C\Delta t+\frac{1}{2}\Big) \Vert \omega_c^m\Vert_{(H^1)'}^2 \nonumber\\
&&\!\!\!\leq \displaystyle\frac{D_c}{8} \Vert \nabla \xi_c^m\Vert_{L^2}^2 + \frac{\Delta t}{24}\Vert \delta_t \xi_c^m\Vert_{L^2}^2+ \frac{1}{2}\Vert  \xi_c^{m-1}\Vert_{L^2}^2 + \Big(\frac{1}{D_c} + \Delta t+1\Big)C\Delta t \!\int_{t_{m-1}}^{t_m}\!\!\!\Vert c_{tt}(t)\Vert_{(H^1)'}^2 dt,
\end{eqnarray}
\begin{eqnarray}\label{Ea1bC}
&J_2+J_{10}&\!\!\!\!\leq  \Vert \xi_c^m\Vert_{L^2} \Vert(\mathcal{I} - \mathbb{P}_c) \delta_t c^m \Vert_{L^2}+ D_c \Vert \xi_c^m\Vert_{L^2} \Vert \theta_c^m\Vert_{L^2}\nonumber\\
&&\!\!\!\! \leq  (\Vert(\mathcal{I} - \mathbb{P}_c) \delta_t c^m \Vert_{L^2}+ D_c \Vert \theta_c^m\Vert_{L^2})(\Delta t \Vert \delta_t \xi_c^m\Vert_{L^2}+ \Vert  \xi_c^{m-1}\Vert_{L^2}) 
\nonumber\\
&&\!\!\!\! \leq  \frac{\Delta t}{24}\Vert \delta_t \xi_c^m\Vert_{L^2}^2+ \frac{1}{2}\Vert  \xi_c^{m-1}\Vert_{L^2}^2+(\Delta t +1)C h^{2(r_2+1)}\Big[\Vert \delta_t c^m\Vert_{H^{r_2+1}}^2+D_c^2 \Vert c^m \Vert_{H^{r_2 +1}}^2\Big] \nonumber\\
&&\!\!\!\!\leq \frac{\Delta t}{24}\Vert \delta_t \xi_c^m\Vert_{L^2}^2+ \frac{1}{2}\Vert  \xi_c^{m-1}\Vert_{L^2}^2\nonumber\\
&&\!\!\!\! \ \ +  (\Delta t +1)C h^{2(r_2+1)}\Big[\displaystyle\frac{1}{\Delta t }\!\int_{t_{m-1}}^{t_m}\!\!\!\Vert c_t \Vert_{H^{r_2+1}}^2  dt+D_c^2 \Vert c^m \Vert_{H^{r_2 +1}}^2\Big],
\end{eqnarray}
\begin{eqnarray}\label{Ea1cC}
&J_4&\!\!\! = A(\xi_{\mathbf{u}}^{m-1},\theta^m_c,\xi^m_c) -A(\mathbb{P}_{\mathbf{u}} \mathbf{u}^{m-1},\theta^m_c,\xi^m_c)  \nonumber\\
&&\!\!\! \leq  \Vert \xi_{\mathbf{u}}^{m-1}\Vert_{L^2}   \Vert \theta_c^m\Vert_{L^\infty\cap W^{1,3}} \Vert \xi_c^m\Vert_{H^1} + \Vert \mathbb{P}_{\mathbf{u}} \mathbf{u}^{m-1}\Vert_{L^\infty\cap W^{1,3}}   \Vert \theta_c^m\Vert_{L^2} \Vert \xi_c^m\Vert_{H^1} \nonumber\\
&&\!\!\!\leq (\Vert \xi_{\mathbf{u}}^{m-1}\Vert_{L^2}   \Vert \theta_c^m\Vert_{L^\infty\cap W^{1,3}}  + \Vert \mathbb{P}_{\mathbf{u}} \mathbf{u}^{m-1}\Vert_{L^\infty\cap W^{1,3}}   \Vert \theta_c^m\Vert_{L^2}) (\Delta t \Vert \delta_t \xi_c^m\Vert_{L^2}+ \Vert  \xi_c^{m-1}\Vert_{L^2}+ \Vert \nabla \xi_c^m\Vert_{L^2})  \nonumber\\
&&\!\!\! \leq \displaystyle\frac{D_c}{8} \Vert \nabla \xi_c^m\Vert_{L^2}^2 + \frac{\Delta t}{24}\Vert \delta_t \xi_c^m\Vert_{L^2}^2+ \frac{1}{2}\Vert  \xi_c^{m-1}\Vert_{L^2}^2 + \Big(\frac{1}{D_c} + \Delta t+1\Big)C \Vert c^m\Vert_{H^2}^2 \Vert \xi_{\mathbf{u}}^{m-1}\Vert_{L^2}^2    \nonumber\\
&&+\Big(\frac{1}{D_c} + \Delta t+1\Big)C h^{2(r_2 +1)} \Vert [\mathbf{u}^{m-1},\pi^{m-1}]\Vert_{H^2\times H^1}^2   \Vert c^m\Vert_{H^{r_2 + 1}}^2,
\end{eqnarray}
\begin{eqnarray}\label{Ea1dC}
&J_3 + J_5&\!\!\!\! \leq (\Vert \mathbf{u}^m\!\!-\!\mathbf{u}^{m-1}\Vert_{L^2}+\Vert \xi_{\mathbf{u}}^{m-1}\Vert_{L^2}+\Vert \theta_{\mathbf{u}}^{m-1}\Vert_{L^2}  )  \Vert c^m\Vert_{L^\infty\cap W^{1,3}} \Vert \xi_c^m\Vert_{H^1}  \nonumber\\
&&\!\!\!\! \leq (\Vert \mathbf{u}^m\!\!-\!\mathbf{u}^{m-1}\Vert_{L^2}+\Vert \xi_{\mathbf{u}}^{m-1}\Vert_{L^2}+\Vert \theta_{\mathbf{u}}^{m-1}\Vert_{L^2}  )  \Vert c^m\Vert_{L^\infty\cap W^{1,3}} (\Delta t \Vert \delta_t \xi_c^m\Vert_{L^2}+ \Vert  \xi_c^{m-1}\Vert_{L^2}+ \Vert \nabla \xi_c^m\Vert_{L^2}) \nonumber\\
&&\!\!\!\! \leq \displaystyle\frac{D_c}{8} \Vert \nabla \xi_c^m\Vert_{L^2}^2 + \frac{\Delta t}{24}\Vert \delta_t \xi_c^m\Vert_{L^2}^2+ \frac{1}{2}\Vert  \xi_c^{m-1}\Vert_{L^2}^2 + \Big(\frac{1}{D_c} + \Delta t+1\Big)C \Vert \mathbf{u}^m-\mathbf{u}^{m-1}\Vert_{L^2}^2  \Vert c^m\Vert_{H^2}^2\nonumber\\
&&+\Big(\frac{1}{D_c} + \Delta t+1\Big)C (h^{2(r+1)} \|[\mathbf{u}^{m-1},\pi^{m-1}]\|_{H^{r+1}\times H^{r}}^2+\Vert \xi_{\mathbf{u}}^{m-1}\Vert_{L^2}^2) \Vert c^m\Vert_{H^2}^2,
\end{eqnarray}
\begin{eqnarray}\label{Ea1eC}
&\displaystyle\sum_{k=6}^8 J_k&\!\!\!\! \leq \gamma \Vert [n^m \!- \!n^{m-1}\!,c^m\! -\! c^{m-1}\!,  \xi^{m-1}_{n},\theta^{m-1}_{n}]\Vert_{L^2}  \Vert[c^m\!, n^{m-1}\!+\!\alpha_0,c^{m-1}]\Vert_{L^\infty} (\Delta t \Vert \delta_t \xi_c^m\Vert_{L^2}+ \Vert  \xi_c^{m-1}\Vert_{L^2})   \nonumber\\
&&\!\!\!\! \leq \frac{\Delta t}{24}\Vert \delta_t \xi_c^m\Vert_{L^2}^2+ \frac{1}{2}\Vert  \xi_c^{m-1}\Vert_{L^2}^2 +(\Delta t + 1)\gamma^2 C \Vert [n^m \!- \!n^{m-1}\!,c^m\! -\! c^{m-1}]\Vert_{L^2}^2 \Vert[c^m\!, n^{m-1}\!+\!\alpha_0,c^{m-1}]\Vert_{H^2}^2\nonumber\\
&&+(\Delta t + 1)\gamma^2C (h^{2(r_1+1)} \|n^{m-1}\|_{H^{r_1+1}}^2+ \Vert \xi^{m-1}_{n}\Vert_{L^2}^2)\Vert[c^m\!, n^{m-1}\!+\!\alpha_0,c^{m-1}]\Vert_{H^2}^2,
\end{eqnarray}
\begin{eqnarray}\label{Ea1fC}
&J_9&\!\!\! \leq \gamma\Vert [\xi^{m-1}_{c}, \theta^{m-1}_{c}]\Vert_{L^2}  \Vert n^{m-1}_h+\alpha_0\Vert_{L^3} (\Delta t \Vert \delta_t \xi_c^m\Vert_{L^2}+ \Vert  \xi_c^{m-1}\Vert_{L^2}+ \Vert \nabla \xi_c^m\Vert_{L^2})  \nonumber\\
&&\!\!\! \leq  \displaystyle\frac{D_c}{8} \Vert \nabla \xi_c^m\Vert_{L^2}^2 + \frac{\Delta t}{24}\Vert \delta_t \xi_c^m\Vert_{L^2}^2+ \frac{1}{2}\Vert  \xi_c^{m-1}\Vert_{L^2}^2\nonumber\\
&& + \Big(\frac{1}{D_c} + \Delta t+1\Big)\gamma^2C (h^{2(r_2+1)} \|c^{m-1}\|_{H^{r_2+1}}^2+\Vert \xi^{m-1}_{c}\Vert_{L^2}^2) \Vert n^{m-1}_h+\alpha_0\Vert_{L^3}^2.
\end{eqnarray}
Therefore, denoting $\gamma_1:=\frac{1}{D_c} + \Delta t+1$ and $\gamma_2:=\Delta t+1$, from (\ref{errc-int2})-(\ref{Ea1fC}), we arrive at
\begin{eqnarray}\label{errcfin}
&\displaystyle\frac{1}{2}&\!\!\!\!\!\delta_t  \Vert \xi_c^m\Vert_{L^2}^2 + \displaystyle\frac{\Delta t}{4} \Vert \delta_t \xi_c^m\Vert_{L^2}^2 + \frac{D_c}{2}\Vert \nabla \xi_c^m\Vert_{L^2}^2  \leq  \gamma_1C\Big(\Delta t \!\int_{t_{m-1}}^{t_m}\!\!\!\Vert c_{tt}(t)\Vert_{(H^1)'}^2 dt + \Vert c^m\Vert_{H^2}^2 \Vert \xi_{\mathbf{u}}^{m-1}\Vert_{L^2}^2\Big) \nonumber\\
&&\!\!\!\!\!\!\!\!\!+ \gamma_2C h^{2(r_2+1)}\Big[\displaystyle\frac{1}{\Delta t }\!\int_{t_{m-1}}^{t_m}\!\!\!\Vert c_t \Vert_{H^{r_2+1}}^2  dt+D_c^2 \Vert c^m \Vert_{H^{r_2 +1}}^2\Big] +\gamma_1 C h^{2(r_2 +1)} \Vert [\mathbf{u}^{m-1}\!,\pi^{m-1}]\Vert_{H^2\times H^1}^2   \Vert c^m\Vert_{H^{r_2 + 1}}^2\nonumber\\
&&\!\!\!\!\!\!\!\!\!+\gamma_1C \Vert \mathbf{u}^m-\mathbf{u}^{m-1}\Vert_{L^2}^2  \Vert c^m\Vert_{H^2}^2+\gamma_1C (h^{2(r+1)} \|[\mathbf{u}^{m-1},\pi^{m-1}]\|_{H^{r+1}\times H^{r}}^2) \Vert c^m\Vert_{H^2}^2\nonumber\\
&&\!\!\!\!\!\!\!\!\! +\gamma_2\gamma^2 C (\Vert [n^m \!- \!n^{m-1}\!,c^m\! -\! c^{m-1}]\Vert_{L^2}^2 + h^{2(r_1+1)} \|n^{m-1}\|_{H^{r_1+1}}^2+ \Vert \xi^{m-1}_{n}\Vert_{L^2}^2)\Vert[c^m\!, n^{m-1}\!+\!\alpha_0,c^{m-1}]\Vert_{H^2}^2\nonumber\\
&&\!\!\!\!\!\!\!\!\!+\gamma_1\gamma^2C (h^{2(r_2+1)} \|c^{m-1}\|_{H^{r_2+1}}^2+\Vert \xi^{m-1}_{c}\Vert_{L^2}^2) \Vert n^{m-1}_h+\alpha_0\Vert_{L^3}^2 + 3\Vert \xi^{m-1}_{c}\Vert_{L^2}^2.
\end{eqnarray}
\vspace{0.5 cm}

{\underline{{\it 3. Error estimate for the velocity} $\mathbf{u}$}}\\

Taking into account the interpolation operator $\mathbb{P}_{\pi}$ defined in (\ref{StokesOp}), we decompose the total error $e_\pi^m$  as follows:
\begin{equation}\label{u1api}
e_\pi^m=(\pi^m -\mathbb{P}_\pi \pi^m) + ( \mathbb{P}_\pi \pi^m - \pi^m_h )=\theta^m_\pi+\xi^m_\pi.
\end{equation}
Then, taking into account (\ref{StokesOp}), from (\ref{erru})-(\ref{errpi}), (\ref{u1a}), (\ref{u1au}) and (\ref{u1api}), we have
\begin{eqnarray}\label{erru-int}
&(\delta_t  \xi_{\mathbf{u}}^m,&\!\!\!\!\!\bar{\mathbf{u}})+ \frac{D_{\mathbf{u}}}{\rho}(\nabla \xi_{\mathbf{u}}^m, \nabla \bar{\mathbf{u}})= (\omega_{\mathbf{u}}^m,\bar{\mathbf{u}})-(\delta_t \theta_{\mathbf{u}}^m,\bar{\mathbf{u}}) - B(\mathbf{u}^m-\mathbf{u}^{m-1},{\mathbf{u}}^m,\bar{\mathbf{u}})   -B(\xi^{m-1}_{\mathbf{u}}+\theta^{m-1}_{\mathbf{u}},{\mathbf{u}}^m,\bar{\mathbf{u}})\nonumber\\
&&\hspace{-0.8 cm} \!-B(\mathbf{u}^{m-1}_h,\xi^m_{\mathbf{u}}+\theta^m_{\mathbf{u}},\bar{\mathbf{u}}) +\frac{1}{\rho}(\xi^m_\pi,\nabla \cdot \bar{\mathbf{u}}) + \frac{1}{\rho}((n^m-n^{m-1})\nabla \phi,\bar{\mathbf{u}})+ \frac{1}{\rho}((\xi^{m-1}_n+\theta^{m-1}_n) \nabla \phi,\bar{\mathbf{u}}),
\end{eqnarray}
\begin{equation}\label{errpi-int}
(\bar{\pi},\nabla \cdot \xi^m_{\mathbf{u}})=0.
\end{equation}
Taking $\bar{\mathbf{u}}=\xi_{\mathbf{u}}^m$ in (\ref{erru-int}), $\bar{\pi}=\frac{1}{\rho}\xi^m_\pi$ in (\ref{errpi-int}), using (\ref{a7}) and adding the resulting expressions, we obtain
\begin{eqnarray}\label{erru-int2}
&\displaystyle\frac{1}{2}&\!\!\!\!\!\delta_t  \Vert\xi_{\mathbf{u}}^m\Vert_{L^2}^2+\displaystyle\frac{\Delta t}{2}\Vert\delta_t \xi_{\mathbf{u}}^m\Vert_{L^2}^2 + \frac{D_{\mathbf{u}}}{\rho}\Vert\nabla \xi_{\mathbf{u}}^m\Vert_{L^2}^2= (\omega_{\mathbf{u}}^m,\xi_{\mathbf{u}}^m)-(\delta_t \theta_{\mathbf{u}}^m,\xi_{\mathbf{u}}^m) - B(\mathbf{u}^m-\mathbf{u}^{m-1},{\mathbf{u}}^m,\xi_{\mathbf{u}}^m)  \nonumber\\
&&\!\!\!\!\! -B(\xi^{m-1}_{\mathbf{u}}\!+\theta^{m-1}_{\mathbf{u}}\!,{\mathbf{u}}^m,\xi_{\mathbf{u}}^m) -B(\mathbf{u}^{m-1}_h\!,\theta^m_{\mathbf{u}},\xi_{\mathbf{u}}^m) + \frac{1}{\rho}((n^m-n^{m-1}+\xi^{m-1}_n+\theta^{m-1}_n) \nabla \phi,\xi_{\mathbf{u}}^m)\nonumber\\
&&\!\!\!\!\! =\sum_{k=1}^{6} L_k.
\end{eqnarray}
Then, using the H\"older and Young inequalities, the Poincar\'e inequality (\ref{PIa}), (\ref{StabStk1})-(\ref{StabStk2}) and  (\ref{aprox01}), we control the terms on the right hand side of (\ref{erru-int2}) as follows
\begin{equation}\label{Ea1au}
L_1\leq \displaystyle\frac{D_{\mathbf{u}}}{10\rho} \Vert \nabla \xi_{\mathbf{u}}^m\Vert_{L^2}^2 + \frac{C\rho}{D_{\mathbf{u}}}\Vert \omega_{\mathbf{u}}^m\Vert_{(H^1)'}^2\leq \frac{D_{\mathbf{u}}}{10\rho} \Vert \nabla \xi_{\mathbf{u}}^m\Vert_{L^2}^2 + \frac{C\rho}{D_{\mathbf{u}}}\Delta t \int_{t_{m-1}}^{t_m}\Vert {\mathbf{u}}_{tt}(t)\Vert_{(H^1)'}^2 dt,
\end{equation}
\begin{eqnarray}\label{Ea1bu}
&L_2&\!\!\!\leq  \Vert \xi_{\mathbf{u}}^m\Vert_{L^2} \Vert(\mathcal{I} - \mathbb{P}_{\mathbf{u}}) \delta_t {\mathbf{u}}^m \Vert_{L^2} \leq \frac{D_{\mathbf{u}}}{10\rho} \Vert \nabla \xi_{\mathbf{u}}^m\Vert_{L^2}^2 + \frac{C\rho}{D_{\mathbf{u}}} h^{2(r+1)}\Vert [\delta_t {\mathbf{u}}^m,\delta_t \pi^m]\Vert_{H^{r+1}\times H^r}^2\nonumber\\
&&\!\!\!\leq \frac{D_{\mathbf{u}}}{10\rho} \Vert \nabla \xi_{\mathbf{u}}^m\Vert_{L^2}^2 + \frac{C\rho}{D_{\mathbf{u}}\Delta t} h^{2(r+1)}\int_{t_{m-1}}^{t_m}\Vert [\mathbf{u}_t,\pi_t] \Vert_{H^{r+1}\times H^r}^2  dt,
\end{eqnarray}
\begin{eqnarray}\label{Ea1cu}
&L_3 + L_4&\!\!\! \leq (\Vert \mathbf{u}^m-\mathbf{u}^{m-1}\Vert_{L^2}+\Vert \xi_{\mathbf{u}}^{m-1}\Vert_{L^2}+\Vert \theta_{\mathbf{u}}^{m-1}\Vert_{L^2}  )  \Vert {\mathbf{u}}^m\Vert_{L^\infty\cap W^{1,3}} \Vert \xi_{\mathbf{u}}^m\Vert_{H^1}  \nonumber\\
&&\hspace{-1.4 cm} \leq\! \frac{D_{\mathbf{u}}}{10\rho} \Vert \nabla \xi_{\mathbf{u}}^m\Vert_{L^2}^2 +\frac{C\rho}{D_{\mathbf{u}}} (\Vert \mathbf{u}^m\!\!-\!\mathbf{u}^{m-1}\Vert_{L^2}^2+h^{2(r+1)} \|[\mathbf{u}^{m-1}\!,\pi^{m-1}]\|_{H^{r+1}\times H^{r}}^2+\Vert \xi_{\mathbf{u}}^{m-1}\Vert_{L^2}^2) \Vert \mathbf{u}^m\Vert_{H^2}^2,
\end{eqnarray}
\begin{eqnarray}\label{Ea1du}
&L_5&\!\!\! = B(\xi_{\mathbf{u}}^{m-1},\theta^m_{\mathbf{u}},\xi^m_{\mathbf{u}}) -B(\mathbb{P}_{\mathbf{u}} \mathbf{u}^{m-1},\theta^m_{\mathbf{u}},\xi^m_{\mathbf{u}})  \nonumber\\
&&\!\!\! \leq  \Vert \xi_{\mathbf{u}}^{m-1}\Vert_{L^2}   \Vert \theta_{\mathbf{u}}^m\Vert_{L^\infty\cap W^{1,3}} \Vert \xi_{\mathbf{u}}^m\Vert_{H^1} + \Vert \mathbb{P}_{\mathbf{u}} \mathbf{u}^{m-1}\Vert_{L^\infty\cap W^{1,3}}   \Vert \theta_{\mathbf{u}}^m\Vert_{L^2} \Vert \xi_{\mathbf{u}}^m\Vert_{H^1} \nonumber\\
&&\!\!\! \leq \ \frac{D_{\mathbf{u}}}{10\rho} \Vert \nabla \xi_{\mathbf{u}}^m\Vert_{L^2}^2 + \frac{C\rho}{D_{\mathbf{u}}}  \Vert [\mathbf{u}^m,\pi^m]\Vert_{H^2\times H^1}^2 \Vert \xi_{\mathbf{u}}^{m-1}\Vert_{L^2}^2 \nonumber\\
&&  + \displaystyle\frac{C\rho}{D_{\mathbf{u}}}h^{2(r +1)} \Vert [\mathbf{u}^{m-1},\pi^{m-1}]\Vert_{H^2\times H^1}^2   \Vert [\mathbf{u}^m,\pi^m]\Vert_{H^{r + 1}\times H^r}^2,
\end{eqnarray}
\begin{eqnarray}\label{Ea1eu}
&L_6&\!\!\! \leq \displaystyle\frac{1}{\rho}\Vert [n^m \!- \!n^{m-1}\!, \xi^{m-1}_{n},\theta^{m-1}_{n}]\Vert_{L^2}  \Vert\nabla \phi\Vert_{L^3} \Vert \xi_{\mathbf{u}}^m\Vert_{H^1}  \nonumber\\
&&\!\!\! \leq \frac{D_{\mathbf{u}}}{10\rho} \Vert \nabla \xi_{\mathbf{u}}^m\Vert_{L^2}^2 + \frac{C}{\rho D_{\mathbf{u}}}  ( \Vert n^m \!- \!n^{m-1}\Vert_{L^2}^2+h^{2(r_1+1)}\! \|n^{m-1}\|_{H^{r_1+1}}^2+ \Vert \xi^{m-1}_{n}\Vert_{L^2}^2)  \Vert\nabla \phi\Vert_{L^3}^2.
\end{eqnarray}
Therefore, from (\ref{erru-int2})-(\ref{Ea1eu}), we arrive at
\begin{eqnarray}\label{errufin}
&\displaystyle\frac{1}{2}&\!\!\!\!\!\delta_t  \Vert\xi_{\mathbf{u}}^m\Vert_{L^2}^2+\displaystyle\frac{\Delta t}{2}\Vert\delta_t \xi_{\mathbf{u}}^m\Vert_{L^2}^2 + \frac{D_{\mathbf{u}}}{2\rho}\Vert\nabla \xi_{\mathbf{u}}^m\Vert_{L^2}^2  \leq  \frac{C\rho}{D_{\mathbf{u}}} \int_{t_{m-1}}^{t_m}\!\!\!\Big(\! \Delta t\Vert {\mathbf{u}}_{tt}(t)\Vert_{(H^1)'}^2 + \frac{h^{2(r+1)}}{\Delta t}\Vert [\mathbf{u}_t,\pi_t] \Vert_{H^{r+1}\times H^r}^2\!\Big)  dt \nonumber\\
&&\!\!\!\!\!\!\!\!\!+\frac{C\rho}{D_{\mathbf{u}}} (\Vert \mathbf{u}^m\!-\!\mathbf{u}^{m-1}\Vert_{L^2}^2+h^{2(r+1)} \|[\mathbf{u}^{m-1}\!,\pi^{m-1}]\|_{H^{r+1}\times H^{r}}^2+\Vert \xi_{\mathbf{u}}^{m-1}\Vert_{L^2}^2) \Vert \mathbf{u}^m\Vert_{H^2}^2\nonumber\\
&&\!\!\!\!\!\!\!\!\!+ \displaystyle\frac{C\rho}{D_{\mathbf{u}}}(\Vert [\mathbf{u}^m,\pi^m]\Vert_{H^2\times H^1}^2 \Vert \xi_{\mathbf{u}}^{m-1}\Vert_{L^2}^2+ h^{2(r +1)} \Vert [\mathbf{u}^{m-1},\pi^{m-1}]\Vert_{H^2\times H^1}^2   \Vert [\mathbf{u}^m,\pi^m]\Vert_{H^{r + 1}\times H^r}^2)\nonumber\\
&&\!\!\!\!\!\!\!\!\! + \frac{C}{\rho D_{\mathbf{u}}}  ( \Vert n^m \!- \!n^{m-1}\Vert_{L^2}^2+h^{2(r_1+1)}\! \|n^{m-1}\|_{H^{r_1+1}}^2+ \Vert \xi^{m-1}_{n}\Vert_{L^2}^2)  \Vert\nabla \phi\Vert_{L^3}^2.
\end{eqnarray}
\vspace{0.5 cm}

{\underline{{\it 4. Error estimate for the flux} $\boldsymbol{\sigma}$}}\\

Then, taking into account (\ref{Interp1New})$_3$, from (\ref{errs}), (\ref{u1a})-(\ref{u1au}) and (\ref{u1ac}), we have
\begin{eqnarray}\label{errs-int}
&(\delta_t \xi_{\boldsymbol\sigma}^m&\!\!\!\!\!,\bar{\boldsymbol \sigma})+ D_c(\nabla \cdot \xi_{\boldsymbol\sigma}^m, \nabla \cdot\bar{\boldsymbol \sigma}) + D_c(\mbox{rot }\xi_{\boldsymbol\sigma}^m, \mbox{rot }\bar{\boldsymbol \sigma})  = (\omega_{\boldsymbol\sigma}^m,\bar{\boldsymbol \sigma})-(\delta_t \theta_{\boldsymbol\sigma}^m,\bar{\boldsymbol \sigma}) \nonumber\\
&&\!\!\!\!\!\! + ((\mathbf{u}^m-\mathbf{u}^{m-1})\cdot {\boldsymbol\sigma}^m+\mathbf{u}^{m-1}({\boldsymbol{\sigma}}^m-{\boldsymbol{\sigma}}^{m-1})+ (\xi^{m-1}_{\mathbf{u}}+\theta^{m-1}_{\mathbf{u}}){\boldsymbol{\sigma}}^{m-1},\nabla \cdot\bar{\boldsymbol \sigma})\nonumber\\
&&\!\!\!\!\!\! + (\mathbf{u}^{m-1}_h(\xi^{m-1}_{\boldsymbol{\sigma}}+\theta^{m-1}_{\boldsymbol{\sigma}})+ \gamma(n^{m}\!-n^{m-1})c^m+\gamma(n^{m-1}\!+\alpha_0)(c^m \!- c^{m-1}) ,\nabla \cdot\bar{\boldsymbol \sigma})
\nonumber\\
&&\!\!\!\!\!\! + \gamma((\xi^{m-1}_n+\theta^{m-1}_n) c_h^{m-1} + (n^{m-1}\!+\alpha_0)(\xi^{m-1}_c+\theta^{m-1}_c),\nabla \cdot\bar{\boldsymbol \sigma}) + D_c(\theta_{\boldsymbol\sigma}^m,\bar{\boldsymbol \sigma}).
\end{eqnarray}
Taking $\bar{\boldsymbol{\sigma}}=\xi_{\boldsymbol\sigma}^m$ in (\ref{errs-int}), we arrive at
\begin{eqnarray}\label{errs-int2}
&\displaystyle\frac{1}{2}&\!\!\!\!\!\delta_t \Vert\xi_{\boldsymbol\sigma}^m\Vert_{L^2}^2+\displaystyle\frac{\Delta t}{2}\Vert\delta_t \xi_{\boldsymbol\sigma}^m\Vert_{L^2}^2+ D_c\Vert \nabla \cdot \xi_{\boldsymbol\sigma}^m\Vert_{L^2}^2+ D_c\Vert\mbox{rot }\xi_{\boldsymbol\sigma}^m\Vert_{L^2}^2 = (\omega_{\boldsymbol\sigma}^m,\xi_{\boldsymbol\sigma}^m)-(\delta_t \theta_{\boldsymbol\sigma}^m,\xi_{\boldsymbol\sigma}^m) \nonumber\\
&&\!\!\!\!\!\! + ((\mathbf{u}^m-\mathbf{u}^{m-1})\cdot {\boldsymbol\sigma}^m+\mathbf{u}^{m-1}({\boldsymbol{\sigma}}^m-{\boldsymbol{\sigma}}^{m-1})+ (\xi^{m-1}_{\mathbf{u}}+\theta^{m-1}_{\mathbf{u}}){\boldsymbol{\sigma}}^{m-1},\nabla \cdot\xi_{\boldsymbol\sigma}^m)\nonumber\\
&&\!\!\!\!\!\! + (\mathbf{u}^{m-1}_h(\xi^{m-1}_{\boldsymbol{\sigma}}+\theta^{m-1}_{\boldsymbol{\sigma}})+ \gamma(n^{m}\!-n^{m-1})c^m+\gamma(n^{m-1}\!+\alpha_0)(c^m \!- c^{m-1}) ,\nabla \cdot\xi_{\boldsymbol\sigma}^m)
\nonumber\\
&&\!\!\!\!\!\! + \gamma((\xi^{m-1}_n+\theta^{m-1}_n) c_h^{m-1} + (n^{m-1}\!+\alpha_0)(\xi^{m-1}_c+\theta^{m-1}_c),\nabla \cdot\xi_{\boldsymbol\sigma}^m) + D_c(\theta_{\boldsymbol\sigma}^m,\xi_{\boldsymbol\sigma}^m)\nonumber\\
&&\!\!\!\!\!\! =\sum_{k=1}^{11} R_k.
\end{eqnarray}
Then, using the H\"older and Young inequalities, the equivalent norm in $\mathbf{H}^1_{\sigma}(\Omega)$ given in (\ref{EQs}), as well as  (\ref{StabStk1})-(\ref{StabStk2}) and (\ref{aprox01})-(\ref{aprox01-a}), we control the terms on the right hand side of (\ref{errs-int2}) as follows
\begin{eqnarray}\label{Ea1aS}
&R_1&\!\!\!\leq (\Vert \xi_{\boldsymbol \sigma}^m\Vert_{L^2}+ \Vert \nabla\cdot \xi_{\boldsymbol \sigma}^m\Vert_{L^2}+\Vert \mbox{rot } \xi_{\boldsymbol \sigma}^m\Vert_{L^2})\Vert \omega_{\boldsymbol \sigma}^m\Vert_{(H^1)'}\nonumber\\
&&\!\!\!\leq  (\Delta t \Vert \delta_t \xi_{\boldsymbol \sigma}^m\Vert_{L^2}+ \Vert  \xi_{\boldsymbol \sigma}^{m-1}\Vert_{L^2}+ \Vert \nabla\cdot \xi_{\boldsymbol \sigma}^m\Vert_{L^2}+\Vert \mbox{rot } \xi_{\boldsymbol \sigma}^m\Vert_{L^2})\Vert \omega_{\boldsymbol \sigma}^m\Vert_{(H^1)'}  \nonumber\\
&&\!\!\!\leq \displaystyle\frac{D_c}{6} \Vert \nabla\cdot \xi_{\boldsymbol \sigma}^m\Vert_{L^2}^2+\displaystyle\frac{D_c}{2}\Vert \mbox{rot } \xi_{\boldsymbol \sigma}^m\Vert_{L^2}^2+ \frac{\Delta t}{8}\Vert \delta_t \xi_{\boldsymbol \sigma}^m\Vert_{L^2}^2+ \frac{1}{2}\Vert  \xi_{\boldsymbol \sigma}^{m-1}\Vert_{L^2}^2 + \Big(\frac{C}{D_c} + C\Delta t+\frac{1}{2}\Big) \Vert \omega_{\boldsymbol \sigma}^m\Vert_{(H^1)'}^2 \nonumber\\
&&\!\!\!\leq \displaystyle\frac{D_c}{6} \Vert \nabla\cdot \xi_{\boldsymbol \sigma}^m\Vert_{L^2}^2+\displaystyle\frac{D_c}{2}\Vert \mbox{rot } \xi_{\boldsymbol \sigma}^m\Vert_{L^2}^2+ \frac{\Delta t}{8}\Vert \delta_t \xi_{\boldsymbol \sigma}^m\Vert_{L^2}^2+ \frac{1}{2}\Vert  \xi_{\boldsymbol \sigma}^{m-1}\Vert_{L^2}^2\nonumber\\
&&\ \ +\gamma_1C\Delta t \!\int_{t_{m-1}}^{t_m}\!\!\!\Vert {\boldsymbol \sigma}_{tt}(t)\Vert_{(H^1)'}^2 dt,
\end{eqnarray}
\begin{eqnarray}\label{Ea1bS}
&\hspace{-0.5cm}R_2+R_{11}&\!\!\!\!\leq  \Vert \xi_{\boldsymbol \sigma}^m\Vert_{L^2} \Vert(\mathcal{I} - \mathbb{P}_{\boldsymbol \sigma}) \delta_t {\boldsymbol \sigma}^m \Vert_{L^2}+ D_c \Vert \xi_{\boldsymbol \sigma}^m\Vert_{L^2} \Vert \theta_{\boldsymbol \sigma}^m\Vert_{L^2}\nonumber\\
&&\!\!\!\! \leq  (\Vert(\mathcal{I} - \mathbb{P}_{\boldsymbol \sigma}) \delta_t {\boldsymbol \sigma}^m \Vert_{L^2}+ D_c \Vert \theta_{\boldsymbol \sigma}^m\Vert_{L^2})(\Delta t \Vert \delta_t \xi_{\boldsymbol \sigma}^m\Vert_{L^2}+ \Vert  \xi_{\boldsymbol \sigma}^{m-1}\Vert_{L^2}) 
\nonumber\\
&&\!\!\!\! \leq \frac{\Delta t}{8}\Vert \delta_t \xi_{\boldsymbol \sigma}^m\Vert_{L^2}^2+ \frac{1}{2}\Vert  \xi_{\boldsymbol \sigma}^{m-1}\Vert_{L^2}^2+\gamma_2C h^{2(r_3+1)}\Big[\Vert \delta_t {\boldsymbol \sigma}^m\Vert_{H^{r_3+1}}^2+D_c^2 \Vert {\boldsymbol \sigma}^m \Vert_{H^{r_3 +1}}^2\Big] \nonumber\\
&&\!\!\!\!\leq \frac{\Delta t}{8}\Vert \delta_t \xi_{\boldsymbol \sigma}^m\Vert_{L^2}^2+ \frac{1}{2}\Vert  \xi_{\boldsymbol \sigma}^{m-1}\Vert_{L^2}^2 + \gamma_2C h^{2(r_3+1)}\Big[\displaystyle\frac{1}{\Delta t }\!\int_{t_{m-1}}^{t_m}\!\!\!\Vert {\boldsymbol \sigma}_t \Vert_{H^{r_3+1}}^2  dt+D_c^2 \Vert {\boldsymbol \sigma}^m \Vert_{H^{r_3 +1}}^2\Big],
\end{eqnarray}

\begin{eqnarray}\label{Ea1eS}
&\displaystyle\sum_{k=3, \ k\neq 6,9}^{10} R_k&\!\!\!\! 
\leq \frac{D_c}{6}\Vert \nabla\cdot \xi_{\boldsymbol \sigma}^m\Vert_{L^2}^2+ \frac{C}{D_c}\Vert [\mathbf{u}^m-\mathbf{u}^{m-1}\!,{\boldsymbol{\sigma}}^m-{\boldsymbol{\sigma}}^{m-1}]\Vert_{L^2}^2  \Vert[{\boldsymbol\sigma}^m\!,\mathbf{u}^{m-1}]\Vert_{L^\infty}^2 \nonumber\\
&&+\frac{C}{D_c} (h^{2(r+1)} \Vert [\mathbf{u}^{m-1},\pi^{m-1}]\Vert_{H^{r+1}\times H^r}^2 + \Vert \xi^{m-1}_{\mathbf{u}}\Vert_{L^2}^2)\Vert{\boldsymbol\sigma}^{m-1}\Vert_{L^\infty}^2\nonumber\\
&&+\frac{\gamma^2 C}{D_c} \Vert [n^m \!- \!n^{m-1}\!,c^m\! -\! c^{m-1}]\Vert_{L^2}^2 \Vert[c^m\!, n^{m-1}\!+\!\alpha_0]\Vert_{L^\infty}^2\nonumber\\
&&+\frac{\gamma^2C}{D_c} (h^{2(r_2+1)} \|c^{m-1}\|_{H^{r_2+1}}^2+ \Vert \xi^{m-1}_{c}\Vert_{L^2}^2)\Vert n^{m-1}\!+\!\alpha_0\Vert_{L^\infty}^2,
\end{eqnarray}
\begin{eqnarray}\label{Ea1e-newS}
& R_6 + R_{9}&\!\!\! =  ((\mathbb{P}_{\mathbf{u}} {\mathbf{u}}^{m-1} - \xi_{\mathbf{u}}^{m-1} )( \xi^{m-1}_{\boldsymbol{\sigma}}+\theta^{m-1}_{\boldsymbol{\sigma}}) +\gamma c^{m-1}_h\xi^{m-1}_{n}+ \gamma(\mathbb{P}_{c}{c}^{m-1} - \xi_{c}^{m-1})\theta^{m-1}_{n},\nabla \cdot\xi_{\boldsymbol\sigma}^m) \nonumber\\
&&\!\!\! \leq \displaystyle\frac{D_c}{6} \Vert \nabla \cdot\xi_{\boldsymbol\sigma}^m \Vert_{L^2}^2+ \frac{C}{D_c} \Vert \mathbb{P}_{\mathbf{u}} {\mathbf{u}}^{m-1}\Vert_{L^\infty}^2 (\Vert \xi^{m-1}_{\boldsymbol{\sigma}}\Vert_{L^2}^2+\Vert \theta^{m-1}_{\boldsymbol{\sigma}}\Vert_{L^2}^2)  +\displaystyle\frac{D_{\mathbf{u}}}{4\rho} \Vert \nabla \xi^{m-1}_{\mathbf{u}}\Vert_{L^2}^{2}  \nonumber\\
&&  +\frac{C\rho}{D_{\mathbf{u}}D_c^2} \Vert \xi^{m-1}_{\mathbf{u}}\Vert_{L^2}^{2}(\Vert \xi^{m-1}_{\boldsymbol{\sigma}}\Vert_{L^4}^{4}+\Vert \theta^{m-1}_{\boldsymbol{\sigma}}\Vert_{L^4}^{4})  + \displaystyle\frac{D_n}{8}\Vert \xi^{m-1}_{n}\Vert_{H^1}^{2} +\frac{C\gamma^4}{D_n D_c^2}\Vert \xi^{m-1}_{n}\Vert_{L^2}^{2} \Vert c^{m-1}_h\Vert_{H^1}^{2} \nonumber\\
&& +\displaystyle\frac{C\gamma^2}{D_c}\Vert \mathbb{P}_{c}{c}^{m-1}\Vert_{L^\infty}^2   \Vert \theta^{m-1}_n\Vert_{L^2}^2+\frac{C\gamma^2}{D_c} \Vert \theta_n^{m-1}\Vert_{L^\infty}^2  \Vert \xi^{m-1}_c\Vert_{L^2}^2\nonumber\\
&&\!\!\! \leq \displaystyle\frac{D_c}{6} \Vert \nabla \cdot\xi_{\boldsymbol\sigma}^m \Vert_{L^2}^2 + \displaystyle\frac{D_{\mathbf{u}}}{4\rho} \Vert \nabla \xi^{m-1}_{\mathbf{u}}\Vert_{L^2}^{2}+ \displaystyle\frac{D_n}{8}\Vert \xi^{m-1}_{n}\Vert_{H^1}^{2} +\frac{C\rho}{D_{\mathbf{u}}D_c^2} \Vert \xi^{m-1}_{\mathbf{u}}\Vert_{L^2}^{2} \nonumber\\
&&+ \displaystyle\frac{C}{D_c} (h^{2(r_3 +1)}\Vert {\boldsymbol{\sigma}}^{m-1}\Vert_{H^{r_3 + 1}}^2+\Vert \xi^{m-1}_{\boldsymbol{\sigma}}\Vert_{L^2}^2) \Vert \mathbf{u}^{m-1}\Vert_{H^2}^2  +\frac{C\gamma^4}{D_n D_c^2}\Vert \xi^{m-1}_{n}\Vert_{L^2}^{2} \Vert c^{m-1}_h\Vert_{H^1}^{2} \nonumber\\
&&+  \displaystyle\frac{C\gamma^2}{D_c}h^{2(r_1 +1)}\Vert {c}^{m-1}\Vert_{H^2}^2  \Vert n^{m-1}\Vert_{H^{r_1 + 1}}^2+\frac{C\gamma^2}{D_c} \Vert n^{m-1}\Vert_{H^2}^2  \Vert \xi^{m-1}_c\Vert_{L^2}^2,
\end{eqnarray}
where, in the first inequality of (\ref{Ea1e-newS}), the 2D interpolation inequality (\ref{in2D}) was used, and in the second inequality, the inductive hypothesis (\ref{IndHyp}) and Lemma \ref{uec} were used. Therefore, from (\ref{errs-int2})-(\ref{Ea1e-newS}), we arrive at
\begin{eqnarray}\label{errsfin}
&\displaystyle\frac{1}{2}&\!\!\!\!\!\delta_t \Vert\xi_{\boldsymbol\sigma}^m\Vert_{L^2}^2+\displaystyle\frac{\Delta t}{4}\Vert\delta_t \xi_{\boldsymbol\sigma}^m\Vert_{L^2}^2+ \frac{D_c}{2}\Vert \nabla \cdot \xi_{\boldsymbol\sigma}^m\Vert_{L^2}^2+ \frac{D_c}{2}\Vert\mbox{rot }\xi_{\boldsymbol\sigma}^m\Vert_{L^2}^2   \leq \gamma_1C\Delta t \!\int_{t_{m-1}}^{t_m}\!\!\!\Vert {\boldsymbol \sigma}_{tt}(t)\Vert_{(H^1)'}^2 dt \nonumber\\
&&\!\!\!\!\!\!\!\!\! +\Vert  \xi_{\boldsymbol \sigma}^{m-1}\Vert_{L^2}^2+ \gamma_2C h^{2(r_3+1)}\Big[\displaystyle\frac{1}{\Delta t }\!\int_{t_{m-1}}^{t_m}\!\!\!\Vert {\boldsymbol \sigma}_t \Vert_{H^{r_3+1}}^2  dt+D_c^2 \Vert {\boldsymbol \sigma}^m \Vert_{H^{r_3 +1}}^2\Big]\nonumber\\
&&\!\!\!\!\!\!\!\!\!+ \frac{C}{D_c}\Big(\Vert [\mathbf{u}^m-\mathbf{u}^{m-1}\!,{\boldsymbol{\sigma}}^m-{\boldsymbol{\sigma}}^{m-1}]\Vert_{L^2}^2  \Vert[{\boldsymbol\sigma}^m\!,\mathbf{u}^{m-1}]\Vert_{L^\infty}^2 +h^{2(r+1)} \Vert [\mathbf{u}^{m-1},\pi^{m-1}]\Vert_{H^{r+1}\times H^r}^2 \Vert{\boldsymbol\sigma}^{m-1}\Vert_{L^\infty}^2\Big)\nonumber\\
&&\!\!\!\!\!\!\!\!\!+\frac{C}{D_c}  \Vert \xi^{m-1}_{\mathbf{u}}\Vert_{L^2}^2\Vert{\boldsymbol\sigma}^{m-1}\Vert_{L^\infty}^2+\frac{\gamma^2 C}{D_c} \Vert [n^m \!- \!n^{m-1}\!,c^m\! -\! c^{m-1}]\Vert_{L^2}^2 \Vert[c^m\!, n^{m-1}\!+\!\alpha_0]\Vert_{L^\infty}^2\nonumber\\
&&\!\!\!\!\!\!\!\!\!+\frac{\gamma^2C}{D_c} (h^{2(r_2+1)} \|c^{m-1}\|_{H^{r_2+1}}^2+ \Vert \xi^{m-1}_{c}\Vert_{L^2}^2)\Vert n^{m-1}\!+\!\alpha_0\Vert_{L^\infty}^2 + \displaystyle\frac{D_{\mathbf{u}}}{4\rho} \Vert \nabla \xi^{m-1}_{\mathbf{u}}\Vert_{L^2}^{2}+ \displaystyle\frac{D_n}{8}\Vert \xi^{m-1}_{n}\Vert_{H^1}^{2}  \nonumber\\
&&\!\!\!\!\!\!\!\!\!+\frac{C\rho}{D_{\mathbf{u}}D_c^2} \Vert \xi^{m-1}_{\mathbf{u}}\Vert_{L^2}^{2}\Vert [{\boldsymbol{\sigma}}^{m-1}_h, {\boldsymbol{\sigma}}^{m-1}]\Vert_{H^1}^{2}+ \displaystyle\frac{C}{D_c} (h^{2(r_3 +1)}\Vert {\boldsymbol{\sigma}}^{m-1}\Vert_{H^{r_3 + 1}}^2+\Vert \xi^{m-1}_{\boldsymbol{\sigma}}\Vert_{L^2}^2) \Vert \mathbf{u}^{m-1}\Vert_{H^2}^2   \nonumber\\
&&\!\!\!\!\!\!\!\!\!+\frac{C\gamma^4}{D_n D_c^2}\Vert \xi^{m-1}_{n}\Vert_{L^2}^{2} \Vert c^{m-1}_h\Vert_{H^1}^{2}+  \displaystyle\frac{C\gamma^2}{D_c}h^{2(r_1 +1)}\Vert {c}^{m-1}\Vert_{H^2}^2  \Vert n^{m-1}\Vert_{H^{r_1 + 1}}^2\nonumber\\
&&\!\!\!\!\!\!\!\!\!+\frac{C\gamma^2}{D_c} \Vert n^{m-1}\Vert_{H^2}^2  \Vert \xi^{m-1}_c\Vert_{L^2}^2.
\end{eqnarray}
\vspace{0.5 cm}

{\underline{\it 5. Estimate for the terms $\Vert n^m - n^{m-1}\Vert_{L^2}$, $\Vert c^m - c^{m-1}\Vert_{L^2}$, $\Vert \mathbf{u}^m - \mathbf{u}^{m-1}\Vert_{L^2}$ and $\Vert{\boldsymbol{\sigma}^m} - \boldsymbol{\sigma}^{m-1}]\Vert_{L^2}$}}\\

Observe that the following estimate holds
\begin{eqnarray}\label{mm1}
&\displaystyle\Delta t \sum_{m=1}^r \Vert [n^m - n^{m-1}&\!\!\!\!,c^m-c^{m-1},\mathbf{u}^m - \mathbf{u}^{m-1},{\boldsymbol{\sigma}^m} - \boldsymbol{\sigma}^{m-1}]\Vert_{L^2}^2\nonumber\\
&&\leq C(\Delta t)^4 \Vert [n_{tt},c_{tt},\mathbf{u}_{tt},{\boldsymbol{\sigma}}_{tt}]\Vert^2_{L^2(L^2)} + C(\Delta t)^2 \Vert [n_{t},c_{t},\mathbf{u}_{t},{\boldsymbol{\sigma}}_{t}\Vert^2_{L^2(L^2)}.
\end{eqnarray}
In fact, 
$$
\Vert \omega_n^m\Vert_{L^2}=\Vert \delta_t n^m - (n^m)_t\Vert_{L^2}=\Big\Vert \frac{1}{\Delta t} (n^m- n^{m-1}) - (n^m)_t\Big\Vert_{L^2}\leq C(\Delta t)^{1/2} \Big(\int_{t_{m-1}}^{t_m}\Vert n_{tt}(t)\Vert_{L^2}^2 dt\Big)^{1/2}, 
$$
where the last inequality was obtained as in (\ref{Ea1a}), with the space norm in $L^2$ instead of $(H^1)',$ the dual of $H^1$. Therefore, we deduce 
$$
\Delta t \sum_{m=1}^r \Vert n^m - n^{m-1}\Vert_{L^2}^2\leq C(\Delta t)^4 \Vert n_{tt}\Vert^2_{L^2(L^2)} + C(\Delta t)^2 \Vert n_t\Vert^2_{L^2(L^2)}.
$$
Analogously, we obtain the estimate for $c, \mathbf{u}$ and ${\boldsymbol{\sigma}}$.

\vspace{0.7 cm}
Then, we can prove the following results:
 \begin{theo}\label{theo1N}
	Under Hypothesis of Lemma \ref{uen}, the following error estimate holds
\begin{equation}\label{EEtheo1}
\|[\xi^m_n,\xi^m_c,\xi^m_{\mathbf{u}},\xi^m_{\boldsymbol{\sigma}}]\|_{l^{\infty}(L^2)\cap l^2(H^1)} \leq C(T) \Big(\Delta t +\max\{h^{r_1+1},h^{r_2+1},h^{r_3+1},
h^{r+1} \}\Big),
\end{equation}
	where the constant $C(T)>0$ is independent of $m, \Delta t$ and $h$.
\end{theo}
\begin{proof}
The proof follows adding the four inequalities (\ref{errnfin}), (\ref{errcfin}), (\ref{errufin}) and (\ref{errsfin}), multiplying the resulting expresion  by $\Delta t$, adding from $m=1$ to $m=s$, using (\ref{mm1}), (\ref{regul}) and Lemmas \ref{uen} - \ref{uec}, and applying the discrete Gronwall Lemma \ref{Diego2} (recalling that $[\xi^0_n,\xi^0_c,\xi^0_{\mathbf{u}},\xi^0_{\boldsymbol{\sigma}}]=[0,0,{\bf 0},{\bf 0}]$). It is clear that the error estimates were derived under the inductive hypothesis (\ref{IndHyp}). Now we have to check it. 
We derive (\ref{IndHyp}) by using (\ref{EEtheo1}) recursively. First, observe that $\Vert \mathbb{P}_{\boldsymbol{\sigma}}\boldsymbol{\sigma}^{m-1}\Vert_{H^1}\leq \Vert {\boldsymbol{\sigma}}\Vert_{L^\infty(H^1)}:=C_0$ for all $m\geq 1,$, from which we deduce that $\Vert {\boldsymbol{\sigma}}^{0}_h\Vert_{H^1}= \Vert \mathbb{P}_{\boldsymbol{\sigma}} {\boldsymbol{\sigma}}_0 \Vert_{H^1}\leq C_0\leq C_0+1:=K$. Now,  notice that
\begin{equation}\label{IndHyp1}
\Vert {\boldsymbol{\sigma}}^{m-1}_h\Vert_{H^1}\leq \Vert \xi^{m-1}_{\boldsymbol{\sigma}}\Vert_{H^1}+\Vert \mathbb{P}_{\boldsymbol{\sigma}}\boldsymbol{\sigma}^{m-1}\Vert_{H^1}.
\end{equation}
Then, it is enough to show that $\Vert  \xi^{m-1}_{\boldsymbol{\sigma}}\Vert_{H^1}\leq 1,$ for each $m\geq 2.$ For that, we consider two cases: First,  if $\frac{(\Delta t)^{1/2}}{h}\leq C,$ then, by using the inverse inequality  $\Vert \xi^m_{\boldsymbol{\sigma}}\Vert_{H^1}\leq h^{-1} \Vert \xi^m_{\boldsymbol{\sigma}}\Vert_{L^2},$ denoting $h_{max}:=\max\{h^{r_1+1},h^{r_2+1},h^{r_3+1},
h^{r+1} \},$ recalling that $h_{max}= h^k,$ (for some $k\geq 2$), and from the calculation of the norm $l^\infty(L^2)$ in (\ref{EEtheo1}), we get 
\begin{eqnarray}\label{j1}
\Vert \xi^1_{\boldsymbol{\sigma}}\Vert_{H^1}\leq \frac{1}{h}\Vert \xi^1_{\boldsymbol{\sigma}}\Vert_{L^2}\leq C(T,\Vert {\boldsymbol{\sigma}}^0_h\Vert_{H^1})\frac{1}{h}(\Delta t+h_{max})\leq  C(T,\Vert {\boldsymbol{\sigma}}^0_h\Vert_{H^1})(C(\Delta t)^{1/2}+h^{k-1}).
\end{eqnarray}
On the other hand, if $\frac{(\Delta t)^{1/2}}{h}$ is not bounded, then $\frac{h_{max}}{\Delta t}\leq \frac{h^2}{\Delta t}\leq C;$ thus, from the calculation of the norm $l^2(H^1)$ in (\ref{EEtheo1}), we get
\begin{eqnarray}\label{j2}
\Vert \xi^1_{\boldsymbol{\sigma}}\Vert^2_{H^1}\leq C(T,\Vert {\boldsymbol{\sigma}}^0_h\Vert_{H^1})\frac{1}{\Delta t}((\Delta t)^{2}+h_{max}^2)\leq  C(T,\Vert {\boldsymbol{\sigma}}^0_h\Vert_{H^1})(\Delta t+C h_{max}).
\end{eqnarray}
In any case, assuming $\Delta t$ and $h$ small enough (without any additional restriction relating the discrete parameters $(\Delta t,h)$), we conclude that $\Vert {\boldsymbol{\sigma}}^{1}_h\Vert_{H^1}\leq K$. Arguing recursively we conclude that $\Vert {\boldsymbol{\sigma}}^{m-1}_h\Vert_{H^1}\leq K$, for all $m\geq 1$.
\end{proof}

 \begin{theo}\label{theo1NCor}
Under Hypothesis of Lemma \ref{uen}, the following error estimate holds
\begin{equation}\label{EEtheo2}
\|\xi^m_{\mathbf{u}}\|_{l^{\infty}(H^1)\cap l^2(W^{1,6})} + \|\xi^m_{\pi}\|_{l^{2}(L^6)} \leq C(T) \Big(\Delta t +\max\{h^{r_1+1},h^{r_2+1},h^{r_3+1},
	h^{r}\}\Big),
\end{equation}
	where the constant $C(T)>0$ is independent of $m, \Delta t$ and $h$.
\end{theo}
\begin{proof}
Applying Lemma 11 of \cite{GV} to (\ref{erru-int})-(\ref{errpi-int}) and usind the interpolation inequality (\ref{in3D}), we have
\begin{eqnarray}\label{errPI1}
&\Vert [\xi_{\mathbf{u}}^m,\xi_{\pi}^m]\Vert_{W^{1,6}\times L^6}^2&\!\!\!\! \leq C (\Vert \delta_t \xi_{\mathbf{u}}^m\Vert_{L^2}^2+ \Vert \omega_{\mathbf{u}}^m\Vert_{L^2}^2+ \Vert \delta_t \theta_{\mathbf{u}}^m\Vert_{L^2}^2+  \Vert \mathbf{u}^{m}-\mathbf{u}^{m-1}\Vert_{L^2}^2 \Vert \nabla \mathbf{u}^{m}\Vert_{L^\infty}^2\nonumber\\
&&+(\Vert \xi_{\mathbf{u}}^{m-1}\Vert_{H^1}^2+\Vert \theta_{\mathbf{u}}^{m-1}\Vert_{H^1}^2) \Vert \mathbf{u}^{m}\Vert_{L^\infty\cap W^{1,3}}^2 + \Vert \xi_{\mathbf{u}}^{m-1}\Vert_{H^1}^2 \Vert \theta_{\mathbf{u}}^{m}\Vert_{L^\infty\cap W^{1,3}}^2\nonumber\\
&&  + \Vert \mathbb{P}_{\mathbf{u}} \mathbf{u}^{m-1}\Vert_{L^\infty\cap W^{1,3}}^2(\Vert \theta_{\mathbf{u}}^{m}\Vert_{H^1}^2+ \Vert \xi_{\mathbf{u}}^{m}\Vert_{H^1}^2)  + \varepsilon(\Vert \xi_{\mathbf{u}}^{m}\Vert_{W^{1,6}}^2 +  \Vert \xi_{\mathbf{u}}^{m-1}\Vert_{W^{1,6}}^2 )\nonumber\\
&& + C_{\varepsilon} \Vert \xi_{\mathbf{u}}^{m}\Vert_{H^1}^2(\Vert \xi_{\mathbf{u}}^{m-1}\Vert_{H^1}^4+\Vert \xi_{\mathbf{u}}^{m-1}\Vert_{H^1}^2\Vert \xi_{\mathbf{u}}^{m}\Vert_{H^1}^2) \nonumber\\
&& + (\Vert n^m- n^{m-1} \Vert_{L^2}^2+\Vert \xi_{n}^{m-1}\Vert_{L^2}^2+\Vert \theta_{n}^{m-1}\Vert_{L^2}^2) \Vert \nabla \phi\Vert_{L^\infty}^2),
\end{eqnarray}
where the constant $C$ depends on the data $(\rho, D_{\mathbf{u}})$. Now, testing (\ref{erru-int}) by $\bar{\mathbf{u}}=\delta_t \xi_{\mathbf{u}}^m$, bounding the terms on the  right hand side as in (\ref{errPI1}) (taking into account that the pressure term vanishes), and adding the resulting expresion with $\lambda$(\ref{errPI1}) (for $0<\lambda<1$), we arrive at 
\begin{eqnarray}\label{errPI2}
& \displaystyle \frac{D_{\mathbf{u}}}{2\rho} &\!\!\!\! \delta_t \Vert \xi_{\mathbf{u}}^{m}\Vert_{H^1}^2 + \Vert \delta_t \xi_{\mathbf{u}}^{m}\Vert_{L^2}^2+ \lambda\Vert [\xi_{\mathbf{u}}^m,\xi_{\pi}^m]\Vert_{W^{1,6}\times L^6}^2 \leq C (\lambda\Vert \delta_t \xi_{\mathbf{u}}^m\Vert_{L^2}^2+ \Vert \omega_{\mathbf{u}}^m\Vert_{L^2}^2+ \Vert \delta_t \theta_{\mathbf{u}}^m\Vert_{L^2}^2\nonumber\\
&&\!\!\!\!+  \Vert \mathbf{u}^{m}-\mathbf{u}^{m-1}\Vert_{L^2}^2 \Vert \nabla \mathbf{u}^{m}\Vert_{L^\infty}^2+(\Vert \xi_{\mathbf{u}}^{m-1}\Vert_{H^1}^2+\Vert \theta_{\mathbf{u}}^{m-1}\Vert_{H^1}^2) \Vert \mathbf{u}^{m}\Vert_{L^\infty\cap W^{1,3}}^2 + \Vert \xi_{\mathbf{u}}^{m-1}\Vert_{H^1}^2 \Vert \theta_{\mathbf{u}}^{m}\Vert_{L^\infty\cap W^{1,3}}^2\nonumber\\
&&\!\!\!\!  + \Vert \mathbb{P}_{\mathbf{u}} \mathbf{u}^{m-1}\Vert_{L^\infty\cap W^{1,3}}^2(\Vert \theta_{\mathbf{u}}^{m}\Vert_{H^1}^2+ \Vert \xi_{\mathbf{u}}^{m}\Vert_{H^1}^2)  + \varepsilon(\Vert \xi_{\mathbf{u}}^{m}\Vert_{W^{1,6}}^2 +  \Vert \xi_{\mathbf{u}}^{m-1}\Vert_{W^{1,6}}^2 )\nonumber\\
&& \!\!\!\! + C_{\varepsilon} \Vert \xi_{\mathbf{u}}^{m}\Vert_{H^1}^2(\Vert \xi_{\mathbf{u}}^{m-1}\Vert_{H^1}^4+\Vert \xi_{\mathbf{u}}^{m-1}\Vert_{H^1}^2\Vert \xi_{\mathbf{u}}^{m}\Vert_{H^1}^2) \nonumber\\
&& \!\!\!\!+ (\Vert n^m- n^{m-1} \Vert_{L^2}^2+\Vert \xi_{n}^{m-1}\Vert_{L^2}^2+\Vert \theta_{n}^{m-1}\Vert_{L^2}^2) \Vert \nabla \phi\Vert_{L^\infty}^2).
\end{eqnarray}
We take $\lambda$ small enough in order to absorb the term $C \lambda\Vert \delta_t \xi_{\mathbf{u}}^m\Vert_{L^2}^2 $ at the right hand side, and $\varepsilon$ small enough with respect to $\lambda$. Moreover, observe that 
\begin{equation}\label{pre2}
\Vert \xi_{\mathbf{u}}^{m-1}\Vert_{H^1}^4+\Vert \xi_{\mathbf{u}}^{m-1}\Vert_{H^1}^2\Vert \xi_{\mathbf{u}}^{m}\Vert_{H^1}^2\leq C.
\end{equation}
In fact, recalling that $h_{max}:=\max\{h^{r_1+1},h^{r_2+1},h^{r_3+1},
h^{r+1} \}$, from estimate (\ref{EEtheo1}) we have in particular that $\|\xi^m_{\mathbf{u}}\|_{H^1}^2 \leq C(T) \Big(\Delta t +\frac{1}{\Delta t}h_{max}^2\Big)$, which implies that
$$
\Vert \xi_{\mathbf{u}}^{m-1}\Vert_{H^1}^4+\Vert \xi_{\mathbf{u}}^{m-1}\Vert_{H^1}^2\Vert \xi_{\mathbf{u}}^{m}\Vert_{H^1}^2 \leq C(T) \Big(\Delta t +\frac{1}{\Delta t}h_{max}^2\Big) \Big((\Delta t)^2 +h_{max}^2\Big).
$$
Therefore, we conclude (\ref{pre2}) under the hypothesis 
\begin{equation}\label{HHa}
\frac{h_{max}^4}{\Delta t}\leq C.
\end{equation}
On the other hand, from (\ref{EEtheo1}) we also have $\|\xi^m_{\mathbf{u}}\|_{L^2}^2 \leq C(T) \Big(\Delta t +h_{max}^2\Big)$ for each $m$. Therefore, by using the inverse inequality $\Vert \xi^m_{\mathbf{u}}\Vert_{H^1}\leq h^{-1} \Vert \xi^m_{\mathbf{u}}\Vert_{L^2}$ we obtain
$$
\Vert \xi_{\mathbf{u}}^{m-1}\Vert_{H^1}^4+\Vert \xi_{\mathbf{u}}^{m-1}\Vert_{H^1}^2\Vert \xi_{\mathbf{u}}^{m}\Vert_{H^1}^2 \leq C(T) \frac{1}{h^4}\Big((\Delta t)^2 +h_{max}^2\Big)^2.
$$
Therefore, we conclude (\ref{pre2}) under the hypothesis 
\begin{equation}\label{HHb}
\frac{(\Delta t)^4}{h^4}\leq C.
\end{equation}
Thus, we conclude (\ref{pre2}) without imposing any restriction on the discrete parameters $(\Delta t,h)$ because for any choice of $(\Delta t,h)$ either (\ref{HHa}) or (\ref{HHb}) holds. Therefore, multiplyng (\ref{errPI2}) by $\Delta t$, adding from $m=1$ to $m=s$, bounding the terms $\Vert \omega_{\mathbf{u}}^m\Vert_{L^2}^2$ and $\Vert \delta_t \theta_{\mathbf{u}}^m\Vert_{L^2}^2$ as in (\ref{Ea1au}) and (\ref{Ea1bu}) respectively, using (\ref{regul}),  (\ref{StabStk1}), (\ref{StabStk2}), (\ref{aprox01})$_1$, (\ref{mm1}), (\ref{EEtheo1}) and taking into account that $\xi^0_{\mathbf{u}}={\bf 0}$, we conclude (\ref{EEtheo2}).
\end{proof}

\begin{remark}
From (\ref{EEtheo1})-(\ref{EEtheo2}), in particular we deduce that $\|[\mathbf{u}^m_h,{\boldsymbol{\sigma}}^m_h]\|_{l^{\infty}(L^2)\cap l^2(H^1)} \leq C(T)$ and \newline{$\Vert [\mathbf{u}^m_h,\pi^m_h]\Vert_{l^{\infty}(H^1)\times l^2(L^6)} \leq C(T)$, for all $m=1,...,N$}.
\end{remark}

\begin{remark}
Theorems \ref{theo1N} and \ref{theo1NCor} in particular imply the convergence of the discrete solutions of the scheme (\ref{scheme1}) towards weak solutions of model (\ref{KNS}), when the parameters $\Delta t$ and $h$ go to $0$.
\end{remark}
	
\section{Some comments on the three-dimensional case}\label{Sub3D}
The results obtained in Subsections \ref{NScheme}-\ref{ESWN}
can be proved also in the three-dimensional case. In fact, observe that the unconditional well-posedness of the scheme (\ref{scheme1}) and the mass-conservation property (\ref{mass01b}) were proved independent on the dimension (see Lemma \ref{MCs} and Theorem \ref{WPs}, respectively). Now, in order to analyze the convergence in the 3D case, we need to modify the inductive hypothesis (\ref{IndHyp}) by the following one:  there exists a positive constant $K>0$, independent of $m$, such that
\begin{equation}\label{IndHyp3D}
\Vert [{\boldsymbol{\sigma}}^{m-1}_h,c^{m-1}_h]\Vert_{H^1}\leq K, \qquad \forall m\geq 1.
\end{equation}
Then, we can prove error estimates for any solution of the scheme (\ref{scheme1}), with respect to a sufficiently regular solution of (\ref{Chemoweak3}). 
 \begin{theo}\label{theo1N3D}
Assume that there exists a sufficiently regular solution of (\ref{Chemoweak3}). If the inductive hypothesis (\ref{IndHyp3D}) is satisfied, then the following error estimates hold
	$$\|[\xi^m_n,\xi^m_c,\xi^m_{\mathbf{u}},\xi^m_{\boldsymbol{\sigma}}]\|_{l^{\infty}(L^2)\cap l^2(H^1)} \leq C(T) \Big(\Delta t +\max\{h^{r_1+1},h^{r_2+1},h^{r_3+1},
	h^{r+1} \}\Big),
	$$
	$$
	\|\xi^m_{\mathbf{u}}\|_{l^{\infty}(H^1)\cap L^2(W^{1,6})} + \|\xi^m_{\pi}\|_{l^{2}(L^6)} \leq C(T) \Big(\Delta t +\max\{h^{r_1+1},h^{r_2+1},h^{r_3+1},
	h^{r}\}\Big),
	$$
	where the constant $C(T)>0$ is independent of $m, \Delta t$ and $h$.
\end{theo}
\begin{proof}
The proof follows as in Theorems \ref{theo1N} and \ref{theo1NCor}, but in this case, in order to bound the terms in the estimates (\ref{Ea1e-new}) and (\ref{Ea1e-newS}), we need to use the 3D interpolation inequality (\ref{in3D}) and the inductive hypothesis (\ref{IndHyp3D}), as follows: 
\begin{eqnarray*}
& I_9&\!\!\! = \chi (\xi^{m-1}_n {\boldsymbol{\sigma}}^{m-1}_h,\nabla \xi_n^m) - \chi(\theta^{m-1}_n \xi_{\boldsymbol{\sigma}}^{m-1},\nabla \xi_n^m)+ \chi(\theta^{m-1}_n \mathbb{P}_{\boldsymbol{\sigma}}{\boldsymbol{\sigma}}^{m-1},\nabla \xi_n^m) \nonumber\\
&&\!\!\! \leq \chi(\Vert \xi^{m-1}_{n}\Vert_{L^2}^{1/2}\Vert \xi^{m-1}_{n}\Vert_{H^1}^{1/2} \Vert{\boldsymbol{\sigma}}^{m-1}_h\Vert_{L^6}+ \Vert \xi_{\boldsymbol\sigma}^{m-1}\Vert_{L^2}   \Vert \theta_n^{m-1}\Vert_{L^\infty}  + \Vert \mathbb{P}_{\boldsymbol{\sigma}} \boldsymbol{\sigma}^{m-1}\Vert_{L^\infty}  \Vert \theta_n^{m-1}\Vert_{L^2}) \Vert \xi_n^m\Vert_{H^1}  \nonumber\\
&&\!\!\! \leq \displaystyle\frac{D_n}{12} \Vert \xi_n^m\Vert_{H^1}^2 +  \displaystyle\frac{D_n}{4} \Vert \xi^{m-1}_{n}\Vert_{H^1}^{2}+ \frac{C\chi^4}{D_n^3} \Vert \xi^{m-1}_{n}\Vert_{L^2}^{2} \nonumber\\
&&+\frac{C\chi^2}{D_n} \Vert n^{m-1}\Vert_{H^2}^2 \Vert \xi_{\boldsymbol{\sigma}}^{m-1}\Vert_{L^2}^2    + \displaystyle\frac{C\chi^2 }{D_n} h^{2(r_1 +1)} \Vert {\boldsymbol{\sigma}}^{m-1}\Vert_{H^2}^2   \Vert n^{m-1}\Vert_{H^{r_1 + 1}}^2,
\end{eqnarray*}
and 
\begin{eqnarray*}
& R_6 + R_{9}&\!\!\! =  ((\mathbb{P}_{\mathbf{u}} {\mathbf{u}}^{m-1} - \xi_{\mathbf{u}}^{m-1} )( \xi^{m-1}_{\boldsymbol{\sigma}}+\theta^{m-1}_{\boldsymbol{\sigma}}) +\gamma c^{m-1}_h\xi^{m-1}_{n}+ \gamma(\mathbb{P}_{c}{c}^{m-1} - \xi_{c}^{m-1})\theta^{m-1}_{n},\nabla \cdot\xi_{\boldsymbol\sigma}^m) \nonumber\\
&&\!\!\! \leq \displaystyle\frac{D_c}{6} \Vert \nabla \cdot\xi_{\boldsymbol\sigma}^m \Vert_{L^2}^2+ \frac{C}{D_c} \Vert \mathbb{P}_{\mathbf{u}} {\mathbf{u}}^{m-1}\Vert_{L^\infty}^2 (\Vert \xi^{m-1}_{\boldsymbol{\sigma}}\Vert_{L^2}^2+\Vert \theta^{m-1}_{\boldsymbol{\sigma}}\Vert_{L^2}^2)  +\displaystyle\frac{D_{\mathbf{u}}}{4\rho} \Vert \nabla \xi^{m-1}_{\mathbf{u}}\Vert_{L^2}^{2}  \nonumber\\
&&  +\frac{C\rho}{D_{\mathbf{u}}D_c^2} \Vert \xi^{m-1}_{\mathbf{u}}\Vert_{L^2}^{2}(\Vert \xi^{m-1}_{\boldsymbol{\sigma}}\Vert_{L^6}^{4}+\Vert \theta^{m-1}_{\boldsymbol{\sigma}}\Vert_{L^6}^{4})  + \displaystyle\frac{D_n}{8}\Vert \xi^{m-1}_{n}\Vert_{H^1}^{2} +\frac{C\gamma^4}{D_n D_c^2}\Vert \xi^{m-1}_{n}\Vert_{L^2}^{2} \Vert c^{m-1}_h\Vert_{L^6}^{4} \nonumber\\
&& +\displaystyle\frac{C\gamma^2}{D_c}\Vert \mathbb{P}_{c}{c}^{m-1}\Vert_{L^\infty}^2   \Vert \theta^{m-1}_n\Vert_{L^2}^2+\frac{C\gamma^2}{D_c} \Vert \theta_n^{m-1}\Vert_{L^\infty}^2  \Vert \xi^{m-1}_c\Vert_{L^2}^2\nonumber\\
&&\!\!\! \leq \displaystyle\frac{D_c}{6} \Vert \nabla \cdot\xi_{\boldsymbol\sigma}^m \Vert_{L^2}^2 + \displaystyle\frac{D_{\mathbf{u}}}{4\rho} \Vert \nabla \xi^{m-1}_{\mathbf{u}}\Vert_{L^2}^{2}+ \displaystyle\frac{D_n}{8}\Vert \xi^{m-1}_{n}\Vert_{H^1}^{2} +\frac{C\rho}{D_{\mathbf{u}}D_c^2} \Vert \xi^{m-1}_{\mathbf{u}}\Vert_{L^2}^{2}\nonumber\\
&&+ \displaystyle\frac{C}{D_c} (h^{2(r_3 +1)}\Vert {\boldsymbol{\sigma}}^{m-1}\Vert_{H^{r_3 + 1}}^2+\Vert \xi^{m-1}_{\boldsymbol{\sigma}}\Vert_{L^2}^2) \Vert \mathbf{u}^{m-1}\Vert_{H^2}^2  +\frac{C\gamma^4}{D_n D_c^2}\Vert \xi^{m-1}_{n}\Vert_{L^2}^{2}  \nonumber\\
&&+  \displaystyle\frac{C\gamma^2}{D_c}h^{2(r_1 +1)}\Vert {c}^{m-1}\Vert_{H^2}^2  \Vert n^{m-1}\Vert_{H^{r_1 + 1}}^2+\frac{C\gamma^2}{D_c} \Vert n^{m-1}\Vert_{H^2}^2  \Vert \xi^{m-1}_c\Vert_{L^2}^2.
\end{eqnarray*}
Finally, the inductive hypothesis (\ref{IndHyp3D}) can be verified in the same spirit of the two-dimensional case (see the proof of Theorem \ref{theo1N}).
\end{proof}
\section{Numerical simulations} 
In this section, we present two numerical experiments: the first one, is used to verify that our scheme gives a good approximation to chemotaxis phenomena in a liquid environment; and the second has been considered in order to check numerically the error estimates proved in our theoretical analysis. All the numerical results are computed by using the software Freefem++. We have considered the spaces for $\eta$, $c$, ${\boldsymbol\sigma}$, $\mathbf{u}$ and $\pi$, generated by $\mathbb{P}_1,\mathbb{P}_1,\mathbb{P}_{1},\mathbb{P}_1-bubble,\mathbb{P}_1$-continuous FE, respectively.\\

\vspace{0.05 cm}
\underline{Test 1:} In this experiment we consider the rectangular domain $\Omega=[0,2] \times [0,1],$ and the initial conditions
\begin{eqnarray*}
\eta_0 &=& \sum_{i=1}^3  \Big(80 \, \text{exp} (-8(x-s_i)^2-10(y-1)^2)\Big),\\
c_0 &=& \, 100 \, \text{exp} (-5(x-1)^2-5(y-0.5)^2),\\
\mathbf{u}_0&=&{\bf 0},
\end{eqnarray*}
where $s_1=0.2$, $s_2=0.5$ and $s_3=1.2$. The numerical solution is computed with mesh parameter $h = 1/40$ and time step $ \Delta t = 1e-5$. Additionally,  we consider the parameters values $\chi=8,$ $D_c=5,$ $\gamma=8,$ $D_{\mathbf{u}}=10$, $D_n=\rho=1$ and $\phi(x,y)=-1000y$. We show the simulations results for the times $t = 0$, $t = 12e -5$ and $t = 30e -5$. The evolution results for the cell density and chemical signal are presented in Figure \ref{fig:NC1}, while the evolution of the velocity field is shown in Figure \ref{fig:U1}. Initially, the cells are in two clusters in the upper part of our domain, and we observe that they begin to orient their movement in the direction of greater concentration of the chemical signal (in this case, the center of the domain). We can see how the clusters of organisms generate a kind of bridge between them and after, we see how organisms tend to agglomerate in the center of the rectangle. Moreover, we can see that the attraction signal $c_h$ is consumed by the organisms, and  some changes in the velocity field are evidenced, influenced by the movement of the cells.

\bigskip

\begin{minipage}{\textwidth}
\begin{tabular}{ccc}
\includegraphics[width=60mm]{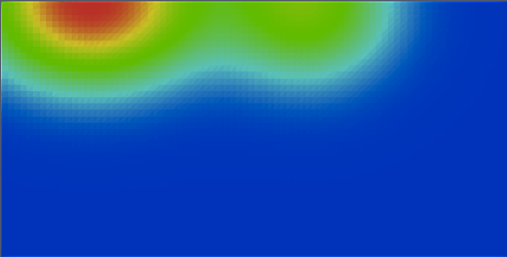} && \includegraphics[width=60mm]{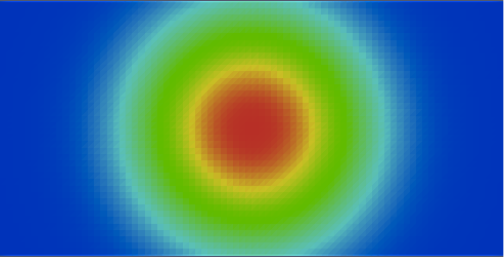}  \\[1mm]
(a) Discrete cell density at time t=0 &&  (b) Discrete cell density at time t=0 \\[2mm]
\includegraphics[width=60mm]{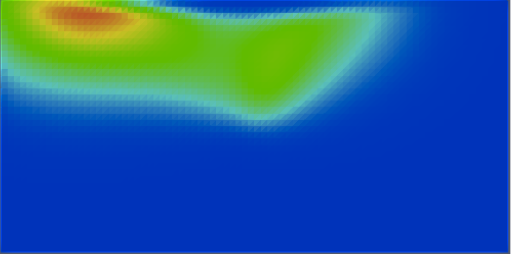} && \includegraphics[width=60mm]{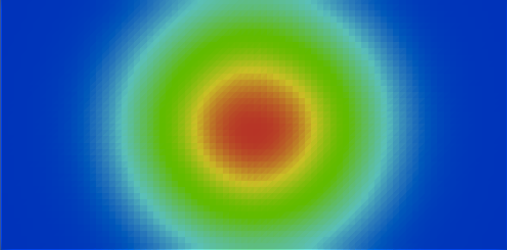}  \\[1mm]
(c) Discrete cell density at time t=12e-5 &&  (d) Discrete chemical signal  at time t=12e-5 \\[2mm]
\includegraphics[width=60mm]{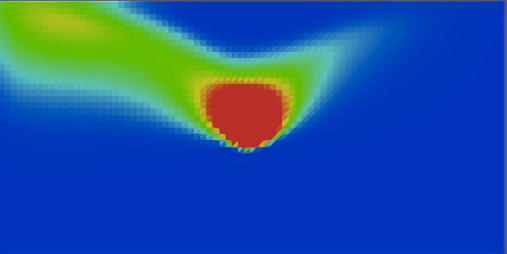} && \includegraphics[width=60mm]{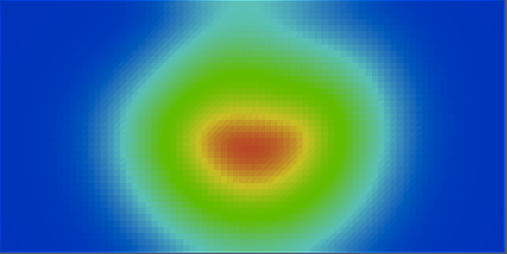}  \\[1mm]
(e) Discrete cell density at time t=30e-5 &&  (f) Discrete chemical signal  at time t=30e-5 \\[-2mm]
\end{tabular}
\figcaption{Cell density vs Chemical concentration.} \label{fig:NC1}
\end{minipage}


\begin{figure}[htbp] 
	\centering 
	\subfigure[Discrete velocity field at time t=1e-5]{\includegraphics[width=60mm]{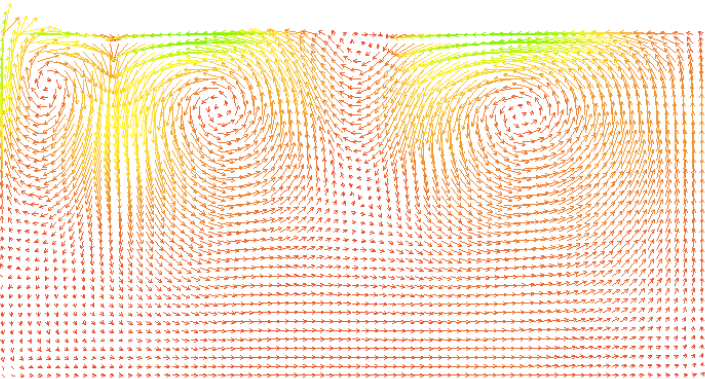}} \hspace{1cm}
	\subfigure[Discrete velocity field at time t=12e-5]{\includegraphics[width=60mm]{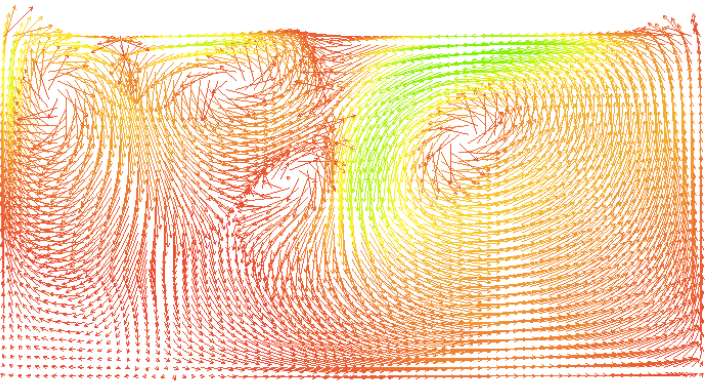}} 
	\subfigure[Discrete velocity field at time t=30e-5]{\includegraphics[width=60mm]{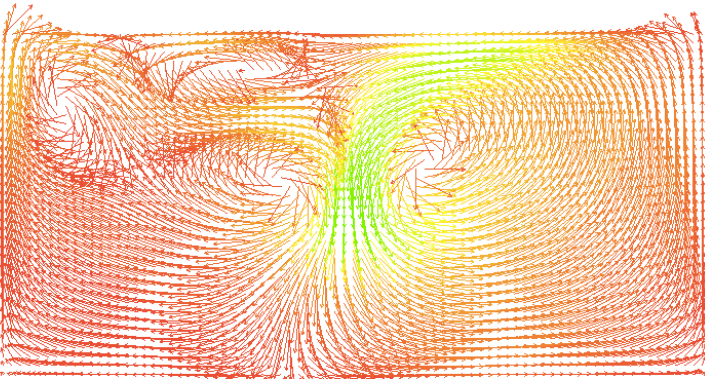}}  \hspace{1cm}
	\caption{Evolution of the velocity field of the fluid.} \label{fig:U1}
\end{figure}


\vspace{0.4 cm}
\underline{Test 2 (Convergence rates):} In this experiment we take $\Omega=[0,1]\times [0,1]$ and we consider the exact solutions $\eta=e^{-t}(cos(2\pi x) +cos(2\pi y) +3)$, $c=e^{-t}(sin(2\pi y)+cos(2\pi x) - 2\pi y +9)$, ${\boldsymbol\sigma}=\nabla c= 2\pi e^{-t} (-sin(2\pi x), cos(2\pi y)-1)$, $\mathbf{u}=e^{-t}(sin(2 \pi y)(-cos(2\pi x + \pi)-1),sin(2 \pi x)(cos(2\pi y + \pi)+1))$ and $\pi=e^{-t}(cos(2\pi x) + sin(2 \pi y))$; and all parameters in (\ref{scheme1}) equal to 1. Note that $\mathbf{u}={\bf 0}$ and $\frac{\partial c}{\partial \boldsymbol{\nu}}=\frac{\partial \eta}{\partial \boldsymbol{\nu}}=0$ on $\partial\Omega$, $\nabla \cdot \mathbf{u}=0$ in $\Omega$ and $\int_{\Omega} \pi =0$. Moreover, we use a uniform partition with $k+1$ nodes in each direction.

Numerical results are presented in Tables \ref{T1}-\ref{T5} for $\Delta t=2e-4$ with respect to time $T=0.01$. We observe the second-order convergence for the total errors in $e^m_{\eta},e^m_{c}, e^m_{\mathbf{u}}$ in $l^{\infty}(L^2)$-norm, and the first-order convergence for  $e^m_{\eta},e^m_{c}, e^m_{\mathbf{u}}$ in $l^{2}(H^1)$-norm and $e^m_{\mathbf{u}}$ in $l^{\infty}(H^1)$-norm, which is
agreement with our theoretical analysis.

\bigskip

\begin{minipage}{\textwidth}
	\begin{footnotesize} 
		\begin{center}
			\begin{tabular}{|| c | c | c | c | c||}
				\hline
				\hline
				$k \times k$ & $\Vert \eta(t_m) - \eta^m_h \Vert_{l^\infty(L^2)}$ & Order & $\Vert \eta(t_m) - \eta^m_h \Vert_{l^2(H^1)}$ & Order  \\
				\hline
				$10 \times 10$ & $5.7265 \times 10^{-2}$  & -   & $1.1682 \times 10^{-1}$  & - \\ 
				\hline
				$20 \times 20$ &  $1.4350 \times 10^{-2}$ & 1.9966 & $5.7520 \times 10^{-2}$ & 1.0223 \\      
				\hline
				$30 \times 30$  &  $6.3060 \times 10^{-3}$   & 2.0279 & $3.8242 \times 10^{-2}$ & 1.0067 \\       
				\hline
				$40 \times 40$  &  $3.4829 \times 10^{-3}$ & 2.0635 & $2.8656 \times 10^{-2}$ & 1.0031 \\
				\hline
				$50 \times 50$  & $2.1757 \times 10^{-3}$   & 2.1085  &  $2.2915 \times 10^{-2}$  & 1.0017 \\    
				\hline
				\hline
			\end{tabular}
		\tabcaption{Convergence rates in space for $\eta$.} 
		\label{T1} 
			\end{center}
	\end{footnotesize} 
\end{minipage}

\begin{minipage}{\textwidth}
	\begin{footnotesize} 
		\begin{center}
			\begin{tabular}{|| c | c | c | c | c||}
				\hline
				\hline
				$k \times k$ & $\Vert c(t_m) - c^m_h \Vert_{l^\infty(L^2)}$ & Order & $\Vert c^m_h - c^m_h \Vert_{l^2(H^1)}$ & Order  \\
				\hline
				$10 \times 10$ & $3.5731 \times 10^{-2}$  & -   & $1.1338 \times 10^{-1}$  & - \\ 
				\hline
				$20 \times 20$ &  $8.9904 \times 10^{-3}$ & 1.9907 & $5.7106 \times 10^{-2}$ & 0.9895 \\      
				\hline
				$30 \times 30$  &  $4.0004 \times 10^{-3}$   & 1.9971 & $3.8126 \times 10^{-2}$ & 0.9964 \\       
				\hline
				$40 \times 40$  &  $2.2512 \times 10^{-3}$ & 1.9986 & $2.8610 \times 10^{-2}$ & 0.9981 \\
				\hline
				$50 \times 50$  & $1.4410 \times 10^{-3}$   & 1.9991  &  $2.2894 \times 10^{-2}$  & 0.9989 \\    
				\hline
				\hline
			\end{tabular}
		\tabcaption{Convergence rates in space for $c$.} 
		\label{T2} 
			\end{center}
	\end{footnotesize} 
\end{minipage}

\begin{minipage}{\textwidth}
	\begin{footnotesize} 
		\begin{center}
			\begin{tabular}{|| c | c | c | c | c||}
				\hline
				\hline
				$k \times k$ & $\Vert \mathbf{u}_1(t_m) - (\mathbf{u}_1)^m_h \Vert_{l^\infty(L^2)}$ & Order & $\Vert\mathbf{u}_1(t_m) - (\mathbf{u}_1)^m_h \Vert_{l^2(H^1)}$ & Order  \\
				\hline
				$10 \times 10$ & $4.1118 \times 10^{-2}$  & -   & $1.5654 \times 10^{-1}$  & - \\ 
				\hline
				$20 \times 20$ &  $1.0106 \times 10^{-2}$ & 2.0245 & $7.7874 \times 10^{-2}$ & 1.0074 \\      
				\hline
				$30 \times 30$  &  $4.4569 \times 10^{-3}$   & 2.0192 & $5.1820 \times 10^{-2}$ & 1.0045 \\       
				\hline
				$40 \times 40$  &  $2.4902 \times 10^{-3}$ & 2.0234 & $3.8827 \times 10^{-2}$ & 1.0034 \\
				\hline
				$50 \times 50$  & $1.5822 \times 10^{-3}$   & 2.0324  &  $3.1043 \times 10^{-2}$  & 1.0027 \\    
				\hline
				\hline
			\end{tabular}
		\tabcaption{Convergence rates in space for $\mathbf{u}_1$ in weak norms.} 
		\label{T3} 
			\end{center}
	\end{footnotesize} 
\end{minipage}

\begin{minipage}{\textwidth}
	\begin{footnotesize} 
		\begin{center}
			\begin{tabular}{|| c | c | c | c | c||}
				\hline
				\hline
				$k \times k$ & $\Vert \mathbf{u}_2(t_m) - (\mathbf{u}_2)^m_h \Vert_{l^\infty(L^2)}$ & Order & $\Vert\mathbf{u}_2(t_m) - (\mathbf{u}_2)^m_h \Vert_{l^2(H^1)}$ & Order  \\
				\hline
				$10 \times 10$ & $4.1175 \times 10^{-2}$  & -   & $1.5655 \times 10^{-1}$  & - \\ 
				\hline
				$20 \times 20$ &  $1.0125 \times 10^{-2}$ & 2.0238 & $7.7875 \times 10^{-2}$ & 1.0075 \\      
				\hline
				$30 \times 30$  &  $4.4658 \times 10^{-3}$   & 2.0190 & $5.1821 \times 10^{-2}$ & 1.0046 \\       
				\hline
				$40 \times 40$  &  $2.4952 \times 10^{-3}$ & 2.0232 & $3.8827 \times 10^{-2}$ & 1.0034 \\
				\hline
				$50 \times 50$  & $1.5855 \times 10^{-3}$   & 2.0322  &  $3.1043 \times 10^{-2}$  & 1.0027 \\    
				\hline
				\hline
			\end{tabular}
		\tabcaption{Convergence rates in space for $\mathbf{u}_2$ in weak norms.} 
		\label{T4} 
			\end{center}
	\end{footnotesize} 
\end{minipage}

\begin{minipage}{\textwidth}
	\begin{footnotesize} 
		\begin{center}
			\begin{tabular}{|| c | c | c | c | c||}
				\hline
				\hline
				$k \times k$ & $\Vert \mathbf{u}_1(t_m) - (\mathbf{u}_1)^m_h \Vert_{l^{\infty}(H^1)}$ & Order & $\Vert\mathbf{u}_2(t_m) - (\mathbf{u}_2)^m_h \Vert_{l^{\infty}(H^1)}$ & Order  \\
				\hline
				$10 \times 10$ & $2.3353 $  & -   & $2.3353$  & - \\ 
				\hline
				$20 \times 20$ &  $1.1882$ & 0.9747 & $1.1882$ & 0.9747 \\      
				\hline
				$30 \times 30$  &  $7.9477 \times 10^{-1}$   & 0.9920 & $7.9477 \times 10^{-1}$ & 0.9920\\       
				\hline
				$40 \times 40$  &  $5.9675 \times 10^{-1}$ & 0.9960 & $5.9675 \times 10^{-1}$ & 0.9960 \\
				\hline
				$50 \times 50$  & $4.7765 \times 10^{-1}$   & 0.9976  &  $4.7765 \times 10^{-1}$  & 0.9976 \\    
				\hline
				\hline
			\end{tabular}
		\tabcaption{Convergence rates in space for $\mathbf{u}_1$ and $\mathbf{u}_2$ in strong norms.} 
		\label{T5} 
			\end{center}
	\end{footnotesize} 
\end{minipage}

\end{document}